\documentclass[a4,12pt]{article}
\usepackage{times}

\usepackage{amsfonts, amssymb,amsmath}
\usepackage{theorem}
\usepackage{hyperref}
\numberwithin{equation}{section}

\newtheorem{theorem}{Theorem}[section]
\newtheorem{proposition}[theorem]{Proposition}
\newtheorem{lemma}[theorem]{Lemma}
\newtheorem{corollary}[theorem]{Corollary}

\newtheorem{example}{Example}[section]

{\theorembodyfont{\rmfamily}
\newtheorem{assumption}{Assumption}[section]
\newtheorem{remark}{Remark}[section]}
\newenvironment{proof}{
  \noindent\textbf{Proof}\ }{\hspace*{\fill}
  \begin{math}\Box\end{math}\medskip}
\newenvironment{proof*}[1]{
  \noindent\textbf{#1\ }}{\hspace*{\fill}
  \begin{math}\Box\end{math}\medskip}

\begin{document}

\def\ssm{\smallsetminus}

\def\1{{\mathbf 1}}
\def\B{{\mathcal B}}
\newcommand{\dd}{\,\text{\rm d}}
\def\R{{\mathbb R}}
\def\F{{\mathcal F}}
\def\G{{\mathcal G}}
\def\C{{{\mathbf C}}}
\def\Q{{{\mathbf Q}}}
\def\N{{{\mathbf N}}}
\def\L{{\mathcal L}}
\def\eqref{{(\ref)}}
\def\I{{\mathcal I}}
\def\l{\ell}
\newcommand{\lra}{\longrightarrow}
\newcommand{\ua}{\uparrow}
\newcommand{\da}{\downarrow}
\newcommand{\dsp}{\displaystyle}
 \newcommand{\Span}{{\mathrm span}}
\def\W{{\mathbf W}}
\def\E{{{\mathbb{E}}}}
\def\p{{{\mathbb P}}}
\def\Z{{{\mathbf Z}}}
\def\loc{\text{\rm{loc}}}
\def \lloc{{L_{loc}}}
\def\det{\mathop{\rm det}}
\def\tr{{\mathop{\rm tr}}}
 \def\bb{\beta}
\def\aa{{\alpha}}
 \def\D{\mathcal D}
\def\e{{\rm{e}}}
\def\ua{\underline a}
\def\OO{\Omega}
 \def\b{\mathbf b}
\def\ee{\varepsilon}
\def\paral{/\kern-0.55ex/}
\def\parals_#1{/\kern-0.55ex/_{\!#1}}
\def\Ric{\mathop{\rm Ric}}
\def\le{{\,\leqslant\,}}
\def\ge{{\geqslant}}
\def\<{\langle}
\def\>{\rangle}
\def\epsilon{{\varepsilon}}

\title{{Strong completeness for a class of stochastic differential equations with irregular coefficients}
\footnote{supported by an EPSRC grant (EP/E058124/1) and
Portuguese Science Foundation
(FCT) for the project ``Probabilistic approach to finite and infinite dimensional
dynamical systems'' (No.\ PTDC/MAT/104173/2008). }}

\author{ {Xin Chen$^{a)}$ and Xue-Mei Li$^{b)}$ }\\
\footnotesize{$^{a)}$Grupo de Fisica Matematica, Universidade de
Lisboa, Av Prof Gama Pinto 2,}\\
\footnotesize{Lisbon 1649-003, Portugal, chenxin\_217@hotmail.com,}\\
\footnotesize{$^{b)}$ Mathematics Institute, The University of Warwick,
Coventry CV4 7AL, U.K.,}\\
\footnotesize{
xue-mei.li@warwick.ac.uk.}}


\date{}
\maketitle

\abstract{We prove the strong completeness for a class of non-degenerate SDEs,
whose coefficients are not necessarily uniformly elliptic nor locally Lipschitz continuous nor bounded.
Moreover,
for each $p>0$ there is a positive number $T(p)$ such that for all $t<T(p)$,
the solution flow $F_t(\cdot)$ belongs to the Sobolev space $W_{\loc}^{1,p}$.  The main tool for this is the
approximation of the associated derivative flow
equations. As an application a differential formula  is  also obtained. }

\section{Introduction}
Throughout the paper $(\Omega,\F, \p)$ is a probability space with complete and right continuous filtration $(\F_t)$, and  $W_t=\{W^1_t, ..., W^m_t\}$ is an $m$-dimensional Brownian motion.
Let $X:\R^m \times \R^d \to \R^d$ be a Borel measurable map such that  for each $x\in \R^d$ the map $X(x, \cdot): \R^m \to \R^d$ is linear and let $X_0: \R^d \to \R^d$ be a Borel measurable vector field on $\R^d$. We study the following SDE,
\begin{equation}\label{sde0}
dx_t=X(x_t)\,dW_t+ X_0(x_t)\,dt.
\end{equation}
Let $X^*(x)$ denote the transpose of $X(x):\R^m \to \R^d$.
We say that the diffusion coefficient  $X$ or the SDE (\ref{sde0}) is  uniformly elliptic if there exists a $\delta>0$ such that  $|(X^*X)(x)(\xi)|\ge \delta |\xi|$ for every $x, \xi \in \R^d$. It is elliptic if $X(x)$ is a surjection for each $x$.

 Fixing an orthonormal basis  $\{e_1,..., e_m\}$ of $\R^m$, for $1 \le k \le m$ and $x\in \R^d$
 we define $X_k(x)=X(x)(e_k)$. Then $\{X_0, X_1, \dots, X_m\}$ is a family of  Borel measurable vector fields on $\R^d$ and the SDE (\ref{sde0}) has the following expression,
\begin{equation}\label{sde}
dx_t=\sum_{k=1}^m   X_k(x_t)\,dW^k_t+ X_0(x_t)\,dt.
\end{equation}
Throughout the paper we assume that there is a unique strong solution to (\ref{sde}) and we denote by  $(F_t(x,\omega), 0\le  t
< \zeta(x,\omega))$
the strong solution with a (non-random) initial value  $x\in \R^d$ and
explosion time $\zeta(x,\omega)>0$. The differential of $X_k$ at $x$ is denoted by $(DX_k)_x$ or $DX_k(x)$.

The SDE (\ref{sde}),  or its solution, is said to be complete if the unique strong solution does not explode, i.e.
$\zeta(x)=\infty$, $\p$-a.s. for every $x \in \R^d$. The  SDE (\ref{sde}), or its solution, is said to be strongly complete if it is complete and
there is a $\p$-null set $\Omega_0$ such that for every   $\omega \notin \Omega_0$,  the function $(t, x)\mapsto
F_t(x,\omega)$ is jointly continuous on $[0,\infty)\times \R^d$.
For further discussion on this, see the books: K. D. Elworthy \cite{Elworthy} and
H. Kunita \cite{Kunita-book}.

If the SDE is strongly complete, the corresponding stochastic dynamics has the perfect cocycle
property, which is often the basic assumption in the study of stochastic dynamical systems.
Continuous dependence on the initial data is also an essential assumption for successful
numerical simulation of the solutions.
It turns out that smoothness and boundedness of the coefficients are not sufficient for the
strong completeness.
In  X.-M. Li and M. Scheutzow \cite{Li-Scheutzow}, a SDE on $\R^2$ of the  form $d x_t=\sigma(x_t, y_t) dB_t$, $dy_t=0$
(here both $x_t$ and $y_t$ are scalar valued process) is constructed with the
property  that
 although $\sigma:\R^2\to \R$ is bounded and  $C^\infty$ smooth, the SDE is not strongly complete.
See also M. Hairer, M. Hutzenthaler and A. Jentzen \cite{Hairer-Hutzenthaler-Jentzen} on the Loss of regularity for Kolmogorov equations.

It is well known, proved by J. N. Blagovescenskii and M. I. Freidlin
\cite{Blagovescenskii-Freidlin61}, that  the SDE (\ref{sde}) is strongly complete  if its coefficients are (globally) Lipschitz continuous.
Suppose that $\{X_k\}_{k=0}^m$ are $C^2$ and $\{DX_k\}_{k=0}^{m}$ are not necessarily bounded,
a sufficient condition for the strong completeness of (\ref{sde}) is given in X.-M. Li \cite{Li-flow}.
In particular, the core condition in \cite{Li-flow} is on the mild growth rate of $\{|DX_k|\}_{k=0}^{m}$,
and the crucial estimate is on the integrability of the norm of
the solution to the derivative flow equation which is controlled by  the growth rate at infinity of the vector fields $\{X_k,$ $ DX_k\}_{k=0}^m$. We would remark that the SDEs studied in \cite{Li-flow} are  on Riemannian manifolds;   specific computations for SDEs on $\R^d$ are given in
\cite[Section 6]{Li-flow}. See also  S. Z. Fang, P. Imkeller and T. S. Zhang \cite{Fang-Imkeller-Zhang} and  X. C. Zhang \cite{Zhang-10}
for different methods to obtain such sufficient conditions.

As mentioned above, a control on the derivatives of the coefficients is useful in estimating the moments of
the solution to the derivative flow equation. The latter also appears to be useful for the study of the
convergence rates  in numerical schemes, see M. Hairer, M. Hutzenthaler and A. Jentzen \cite{Hairer-Hutzenthaler-Jentzen}, where
they construct some  SDEs with smooth bounded coefficients whose solutions fall into one of the following cases: (1) the map $x\mapsto \E(F_t(x))$ is continuous
but not locally H\"older continuous; (2) for any $t\ge 2$,  $C>0$, $\alpha>0$, and  $h_0>0$, there is
a step size $h\in (0, h_0)$ with the property that the rate of convergence for the Euler-Maruyama method is slower
 than $Ch^\alpha$.

Let us consider the case that the coefficients of SDE (\ref{sde}) are not Lipschitz continuous.
If $X$ is uniformly elliptic, $\{X_k\}_{k=0}^m$ are bounded, and $X_k\in W^{1,2d}_{\loc}(\R^d; \R^d)$ for each $k\ge 1$, it is  established in
A. Veretennikov \cite{Veretennikov80} that there is a  unique strong solution to (\ref{sde}).
Letting $m=d$ and $X(x)$ be the identity matrix, in \cite{Krylov-Rockner},
N. V. Krylov and M. R\"ockner  prove that  there is a unique global strong solution  provided that
 $X_0\in L^q([0,T];L^p(\R^d;\R^d))$
for some $p>1$, $q>2$ satisfying the condition $\frac{d}{p}+\frac{2}{q}<1$. The strongly completeness for such SDE is obtained  by
E. Fedrizzi and F. Flandoli \cite{Fedrizzi-Flandoli}. See also related works by
I. Gy\"ongy and T. Martinez \cite{Gyongy-Martinez} and A. M. Davie \cite{Davie07}.
Similar results hold for the multiplicative noise case:  suppose that $X$ is uniformly elliptic
and uniformly continuous with $|DX_k| \in L^q([0,T];L^p(\R^d)),\ 1\le k \le m$, $|X_0| \in $
$ L^q([0,T];L^p(\R^d))$ for $p,q$ as above,  then (\ref{sde}) is  shown to be strongly complete by
X. C. Zhang \cite{Zhang-05,Zhang-11}.
If $X$ is uniformly elliptic,  $X_0 \in C^{0,\delta}(\R^d;\R^d)$ and $\{X_k\}_{k=1}^m \subseteq$ $ C_b^{3,\delta}(\R^d;\R^d)$
 for some $0<\delta<1$, it is proved by F. Flandoli, M. Gubinelli and E. Priola
\cite{Flandoli-Gubinelli-Priola}
that (\ref{sde}) is strongly complete and the solution flow $F_t(\cdot,\omega)$ is differentiable
with respect to the space variable.
For bounded measurable drifts, see also the Ph.D. thesis of X. Chen \cite{Chen} and a recent paper of
S. E. A. Mohammed, T. Nilsen and F. Proske \cite{MNP} where the the noise is essentially additive and the
solution flow of (\ref{sde}) is shown to belong to a (weighted) Sobolev space,
which generalises the result in N. Bouleau and F. Hirsch \cite{Bouleau-Hirsch} where the coefficients are Lipschitz continuous.
We also refer
 to readers to S. Z. Fang and T. S. Zhang \cite{Fang-Zhang}, S. Z. Fang and D. J. Luo \cite{Fang-Luo}, and
 S. Cox, M. Hutzenthaler and A. Jentzen \cite{Cox-Hutzenthaler-Jentzen}
on the study of strong completeness for a SDE whose coefficients are not (locally) Lipschitz continuous nor elliptic.

In all the results mentioned earlier, concerning with the strong completeness of a SDE whose
coefficients are not restricted to the class of (locally) Lipschitz continuous vector fields,
some uniform  conditions are assumed,
such as the uniform continuity condition, or the $L^{p}$ integrability, or the uniform ellipticity, which are quite
different from the mild growth conditions in \cite{Fang-Imkeller-Zhang}, \cite{Li-flow}, \cite{Zhang-10} for the SDEs with
locally Lispchitz continuous coefficients.
 In this paper we are specially interested in SDEs whose coefficients are not locally Lipschitz continuous nor necessarily
satisfying some uniform conditions.



Some preliminary results in this paper appeared in our earlier work,  \cite{Chen-Li-Arxiv}, we have strengthened the results there
by removing the boundedness condition and the uniform ellipticity condition on the diffusion coefficients.

Throughout this paper the components of the vector fields $X_k$ are denoted by $X_k=( X_{k1}, \dots,  X_{kd})$, $0\le k \le m$.
Let  $X^*X=(a_{i,j})_{i,j=1}^d $ be the $d\times d$ diffusion matrix with entries
$a_{i,j}(x)=\sum_{k=1}^m X_{ki}(x) X_{kj}(x)$.

For $x,\xi \in \R^d$, let
\begin{equation*}
H_{p}(x)\big(\xi,\xi\big):=2p\big\langle D X_0(x)(\xi), \xi\big\rangle
+(2p-1)p\sum_{k=1}^{m}\big|D X_k(x)(\xi)\big|^2
\end{equation*}
and we define the real valued function
\begin{equation}\label{kp}
K_p(x):=\sup_{|\xi|=1}\left\{H_{p}(x)\big(\xi,\xi\big)\right\}.
\end{equation}

\begin{assumption}\label{assumption1}
\begin{enumerate}
\item [(1)]  There exist positive constants $p_1, C_1$, such that,
\begin{equation}\label{c1}
\sum_{i,j=1}^d a_{i,j}(x)\xi_i \xi_j\geqslant \frac{C_1}{1+|x|^{p_1}} |\xi|^2,
\quad  \forall  x \in \R^d,\ \xi=(\xi_1,...,\xi_d)\in \R^d.
\end{equation}
\item[(2)] There exist positive constants $C_2$, $p_2$, such that for
all $0\le k \le m$,
\begin{equation}\label{c2aa}
|X_k(x)|\le C_2(1+|x|^{p_2}).
\end{equation}
There is a constant $0<\delta\le 1$, such that for every $p>0$,
\begin{equation}\label{c2}
\sup_{|y|\le \delta}\Big(\sum_{k=1}^m p|X_k(x+y)|^2+\langle x , X_0(x+y)\rangle \Big) \le C(p)(1+|x|^2)
\end{equation}
for some positive constant $C(p)>0$.
\item[(3)]  There are constants $p_3>2(d+1)$, $p_4>d+1$ such that
$X_k \in W_{\loc}^{1,p_3}(\R^d;\R^d)$, $1\le k \le m$ and $X_0 \in W_{\loc}^{1,p_4}(\R^d;$ $\R^d)$.
For every $p>1$, there exists a constant $\kappa(p)>0$, such that for every $R>0$,
\begin{equation}\label{c3}
\int_{\{|x|\le R\}} e^{\kappa(p) K_{p}(x)} dx<\infty.
\end{equation}
Here $K_p(x)$ is defined by (\ref{kp}).

\item[(4)] There exist positive constants $R_1$, $C_3$, $p_5$, such that
for all $0 \le k \le m $
\begin{equation}\label{c4}
|DX_k(x)|\le C_3(1+|x|^{p_5}),\quad \forall \; |x|>R_1.
\end{equation}
For every $p>0$, there exists a constant $C(p)>0$, such that,
\begin{equation}
K_p(x)\le  C(p)\log (1+|x|^2), \quad  \forall \;|x|>R_1.
\label{c4aa}
\end{equation}

  \end{enumerate}
\end{assumption}

\medskip

The main theorem of the paper is as following:
\begin{theorem}\label{th:regularity}
Under Assumption \ref{assumption1} the SDE (\ref{sde}) is strongly complete.
Furthermore, for every $p>0$  there is a positive constant $T_1(p)$ such that for each $t\in [0, T_1]$,
$F_t(\cdot,\omega) \in W^{1,p}_{\loc}(\R^d; \R^d)$, $\p$-a.s..
\end{theorem}


We  comment on  Assumption \ref{assumption1}.
Condition (\ref{c2})  is a technical condition that is used for approximating
(\ref{sde1}) by a family of SDEs with smooth coefficients  satisfying
\begin{equation}\label{e0-1}
p\sum_{k=1}^m |X_k(x)|^2+\langle x , X_0(x)\rangle \le C(p)(1+|x|^2).
\end{equation}
A SDE with coefficients satisfying condition (\ref{e0-1})  is complete, see  e.g.  \cite{Li-flow}.
The constant $\kappa(p)$ in (\ref{c3}) is allowed to decrease with $p$.
In conditions (\ref{c2}-\ref{c3}), the restrictions on $X_0$ are only one-sided.
In particular condition (\ref{c3}) does not imply that
$\exp\big(p |DX_k|^2\big)$ is locally integrable. In fact,
if $\sup_{|\xi|=1}\langle DX_0(\xi), \xi\rangle$ is  negative enough, it compensates
the contribution of the norms of the derivatives of the diffusion coefficients to $K_p$, c.f. Example \ref{ex1} below.
The  conditions (1) and (3) of Assumption \ref{assumption1} imply that
there is a unique strong solution to (\ref{sde}). Indeed
 since $X$ is elliptic,  $X_k \in W_{\loc}^{1,p_3}(\R^d;\R^d)$ for $1 \le k \le d$,
and $X_0 \in W_{\loc}^{1,p_4}(\R^d;\R^d)$, we may apply  \cite[Theorem 1.3]{Zhang-11}.
Moreover, under condition (\ref{c2}), the SDE (\ref{sde})  is complete. Roughly speaking, Assumption
\ref{assumption1} means that the coefficients are contained in some Sobolev space and satisfy some
local integrability condition in a compact set, in particular, the coefficients may not be Lipschitz
continuous in this compact set,  while outside such compact set, the mild growth rate for
the derivatives of
the coefficients are needed.

We also comment on the  proof of the theorem.  In
N. V. Krylov and M. R\"ockner \cite{Krylov-Rockner} and X. C. Zhang \cite{Zhang-05, Zhang-11},  a transformation,
first introduced in A. K. Zvonkin \cite{Zvonkin74}, are applied to transform
(\ref{sde}) to  a  SDE  without drift. In order to apply the Zvonkin transformation,  global estimates for the solution to the associated parabolic PDE are required.
Such estimates are usually obtained under the assumption that the diffusion coefficients
are uniformly elliptic and uniformly continuous, see e.g. N. V. Krylov \cite{Krylov-book}.
In Assumption \ref{assumption1}, we do not assume the diffusion coefficients to be uniformly elliptic or to be uniformly continuous,
nor the derivatives satisfy some $L^p$ integrability conditions,
no suitable estimates for the corresponding PDE  is available. We  therefore have to assume
the drift coefficients to be more regular than that in the reference mentioned above.

We adapt the philosophy in \cite{Li-flow} and study the strongly completeness of (\ref{sde})
by investigating the corresponding derivative
flow equation. But the methods here are however quite different due to the irregularity of 
the coefficients. In fact, the derivative 
flow equation is
\begin{equation}\label{sde1}
\begin{cases}
&dx_t=\sum_{k=1}^m X_k(x_t)dW_t^k+X_0(x_t)dt, \\
&dv_t=\sum_{k=1}^m D X_k(x_t)(v_t)dW_t^k+D X_0(x_t)(v_t)dt.
\end{cases}
\end{equation}
Here $v_t$ is a $\R^{d}$-valued process.
 Since the coefficients $\{X_k\}_{k=0}^m$  are not necessarily locally Lipschitz continuous,
 at this stage, the derivative flow equation, whose coefficients
are not necessarily locally bounded, is only a formal expression.
We must establish firstly the  pathwise uniqueness and the existence of a strong solution
to the derivative flow equation (\ref{sde1}).

Let $(F_t(x), V_t(x,v))$ be the strong solution to (\ref{sde1}) with initial point $x_0=x \in \R^d$, $v_0=v \in \R^{d}$.
In case of  $\{X_k\}_{k=0}^m$ belonging to $C_b^2(\R^d;\R^d)$, it is well known that
 $D_x F_t(x)(v)=V_t(x,v)$ $\p-a.s.$, see. e.g. H. Kunita \cite{Kunita-book}.  In this paper, we use
 the approximating Theorem (Theorem \ref{th1}) to establish such a result,  c.f. Theorem \ref{th:regularity}.
Furthermore letting $DX_k$ and $\tilde DX_k$ be  two different version of the weak derivative of
$X_k$, we show that
\begin{equation*}
\int_0^T |DX_k(x_t)-\tilde DX_k(x_t)|^2 dt=0,\ \ \p-a.s..
\end{equation*}
It follows that the It\^o integral $\int_0^T DX_k(x_t) (v_t)dW_t^k$ is independent of the choice of versions of
$DX_k$.

The remaining part of the paper is organized as following. In section \ref{sec-example}, we give an example of a SDE
 which satisfies Assumption \ref{assumption1}. This example is not covered by the reference listed above.
 In Section \ref{sec-convergence} we establish a lemma for the approximation of a strong solution to a SDE with
 pathwise uniqueness property.
 Section \ref{sec-distribution} is devoted to an estimation for the distribution of the solution to (\ref{sde}). 
 A key step in the proof of the main theorem is presented in Section \ref{sec-construction}, where we construct an approximating
 sequence of smooth vector fields $\{X_k^\ee\}_{k=0}^m$, which satisfy
the conditions of Assumption \ref{assumption1} with the corresponding constants independent of $\ee$.
 In Section \ref{sec-derivative-flow} we give uniform estimates on the approximating derivative flow equations.
The key convergence  result is presented in Theorem \ref{th1}.
In section \ref{sec-proof} we complete the proof of the main theorem.
Finally  a differentiation formula is  established in Section \ref{bel-formula}.

\medskip

Notation. The symbol $C$ denotes a constant that may vary in different places and depend only on dimension
$d$ and the constants in Assumption \ref{assumption1}. If it depends on another
parameter, it will be emphasized by an index.


%


\section{An Example }
\label{sec-example}
The example below satisfies Assumption \ref{assumption1}, as far as we know  it is not covered by
results from the existing literature. The vector fields $\{X_k\}_{k=1}^d $ constructed below are
not uniformly elliptic  if $q_2<0$;  while $\{X_k\}_{k=1}^d $
 are not bounded nor uniformly continuous if $q_2>0$.

\begin{example}\label{ex1}
We suppose that $q_1, q_3, q_4$ are positive numbers and $q_2\in \R$.
For a fixed orthonormal basis $\{e_1, \dots, e_d\}$ of $\R^d$ and $ 1\le k \le d$  we define
\begin{align*}
X_k(x)&=\Big((1+|x|^{q_1})g_1(x)+|x|^{q_2}g_2(x)\Big)e_k,\\
X_0(x)&=\Big(-(1+|x|^{-q_3})g_1(x)-|x|^{q_4}g_2(x)\Big)x,
\end{align*}
where  $g_1, g_2 $ are $C^\infty$ functions on $\R^d$ with the following specifications
\begin{equation*}
g_1(x)=\begin{cases}
& 1,\qquad \text{if}\ \ |x|\le 2,\\
& \in [0,1], \quad \text{if}\ \ 2<|x|< 3,\\
& 0,  \qquad  \text{if}\ \ |x|\ge 3,
\end{cases}
\end{equation*}
\begin{equation*}
g_2(x)=\begin{cases}
& 0, \qquad \text{if}\ \ |x|\le 1,\\
& \in [0,1], \quad \text{if}\ \ 1<|x|< 2,\\
& 1,\qquad \text{if}\ \ |x|\ge 2.
\end{cases}
\end{equation*}
Suppose that the constants $q_1,q_2, q_3$ and $ q_4$  satisfy the following relations:
$$q_4+2>2q_2, \quad 1-\frac{d}{2(d+1)}<q_1<1,  \quad 2(1-q_1)<q_3<\frac{d}{d+1}.$$
 Then $\{X_k\}_{k=0}^d$  satisfy Assumption \ref{assumption1} and
the corresponding SDE (\ref{sde}) is strongly complete.
\end{example}

We first check the ellipticity condition. If $q_2\ge 0$,
\begin{equation*}
\sum_{i,j=1}^d a_{i,j}(x)\xi_i\xi_j \ge |\xi|^2,\ \ \forall\ \xi=(\xi_1,\dots, \xi_d)\in \R^d.
\end{equation*}
If $q_2<0$,
\begin{equation*}
\sum_{i,j=1}^d a_{i,j}(x)\xi_i\xi_j \ge \frac{C |\xi|^2}{1+|x|^{-q_2}},\ \ \forall\  \xi=(\xi_1,\dots, \xi_d)\in \R^d.
\end{equation*}
In both cases (\ref{c1}) is true.

It is obvious that (\ref{c2aa}) holds,  and for $|x|$ sufficiently large,
\begin{align*}
&\sup_{|y|\le 1}\Big(\sum_{k=1}^d p|X_k(x+y)|^2+\langle x , X_0(x+y)\rangle \Big)\\
&\le C(p)|x|^{2q_2}+\sup_{|y|\le 1}\Big(-|x+y|^{q_4}\langle x, (x+y)\rangle\Big)\\
&\le C(p)|x|^{2q_2}-C(|x|^{q_4}-1)|x|^2+C\sup_{|y|\le 1}\big(|x|^{q_4+1}|y|\big)\\
&\le -C(p)(1+|x|^{q_4+2}),\end{align*}
where the last step is due to the assumption $q_4+2>2q_2$. We have proved (\ref{c2}).

We prove below that $X_k \in W_{\loc}^{1,p}(\R^d;\R^d)$. Firstly for every $1\le k \le d$, $X_k$ is smooth
on $\R^d \backslash \{0\}$, we only need to consider the domain
$\{x \in \R^d; 0<|x| \le 1\}$.  Let $\otimes$ denote the tensor product operator and let $\mathbf{I}:\R^d \to \R^d$ denote
identity map. For all $x \in \R^d$ with $0<|x|\le 1$,
\begin{equation}\label{ex1-1}
\begin{split}
DX_k(x)&=q_1|x|^{q_1-2}e_k\otimes x,\\
DX_0(x)&=q_3|x|^{-q_3-2}x \otimes x-(1+|x|^{-q_3})\mathbf{I}.
\end{split}
\end{equation}
So for every $x \in \R^d$ with $0<|x|\le 1$,
\begin{equation*}
\begin{split}
|DX_k(x)|\le C|x|^{q_1-1},\ \ |DX_0(x)|\le C|x|^{-q_3}.
\end{split}
\end{equation*}
The condition $q_3<\frac{d}{d+1}$ and $0<1-q_1<\frac{d}{2(d+1)}$ ensure that, for $1\le k \le m$,
$X_k$ belongs to  $ W^{1,p_3}_{\loc}(\R^d;$ $\R^d)$ and  $X_0$ belongs to $ W^{1,p_4}_{\loc}(\R^d;\R^d)$
for some constants $p_3$ and $p_4$ satisfying the following relations
$$2(d+1)<p_3<\frac{d}{1-q_1}, \qquad d+1<p_4<\frac{d}{q_3}.$$

For the local exponential integrability, (\ref{c3}), we again only need to consider  the domain
$\{x \in \R^d; 0<|x| \le 1\}$.
From (\ref{ex1-1}) we know that,  $$\sup_{|\xi|=1}\langle DX_0(x)\xi, \xi\rangle \le
-(1-q_3)|x|^{-q_3} \quad \forall\  0<|x|\le 1.$$ Therefore for $|x|$ small enough,
\begin{equation*}
K_p(x)\le C(p)|x|^{-2(1-q_1)}-C|x|^{-q_3}\le -C(p)|x|^{-q_3} \le 0,
\end{equation*}
where we use condition $q_3 >2(1-q_1)$. Hence (\ref{c3}) holds.

If $|x|>3$,
\begin{equation*}
\begin{split}
&DX_k(x)=q_2|x|^{q_2-2}e_k\otimes x ,\ \ \ \ 1\le k \le d,\\
&DX_0(x)=-(1+|x|^{q_4})\mathbf{I}-q_4|x|^{q_4-2}x\otimes x.
\end{split}
\end{equation*}
(\ref{c4}-\ref{c4aa}) of
Assumption \ref{assumption1} follows from $q_4+2>2q_2$.

\section{A convergence Lemma}\label{sec-convergence}
Let $Y_k^{\varepsilon}\in C^{\infty}(\R^d;\R^d)$, $0\le k \le m$, $\ee \in (0,\ee_0)$
be a family of smooth vector fields, where  $\ee_0$ is a positive constant. We consider the
following SDE
\begin{equation}\label{e3}
\begin{split}
dy_t^{\ee}=\sum_{k=1}^m Y_k^{\ee}(y_t^{\ee})\,dW_t^k+Y_0^{\ee}(y_t^{\ee})\,dt.
\end{split}
\end{equation}
Since each $Y_k^{\varepsilon}$ is smooth it is well known that
 (\ref{e3}) has a  unique maximal strong solution. Throughout this section
  we also assume that  (\ref{e3}) is complete  for each $\epsilon \in (0,\ee_0)$  and we denote by $(\phi^{\ee}_t(x))$  its
strong solution  with initial point $x \in \R^d$.

 Let $\{Y_k\}_{k=0}^m$ be Borel measurable vector fields on $\R^d$. Now we do not
 assume any regularity assumption on the vector fields $\{Y_k\}_{k=0}^m$
 and then have no information on  the existence or the uniqueness of a strong solution to the following SDE
\begin{equation}\label{e4}
\begin{split}
dy_t=\sum_{k=1}^m Y_k (y_t)\,dW_t^k+Y_0(y_t)\,dt.
\end{split}
\end{equation}
One well known method for the existence of  a strong solution is the
Watanabe-Yamada method:  if there is a weak solution and  the pathwise uniqueness holds for  SDE (\ref{e4}), then
 there exists a unique strong solution to (\ref{e4}), see e.g. \cite{Ikeda-Watanabe}.


In Lemma \ref{lem5} we prove that under suitable conditions, the solutions of (\ref{e3}) converges to
the unique strong solution to (\ref{e4}).
As pointed in N. V. Krylov and A. K. Zvonkin \cite{Krylov-Zvonkin}, and H. Kaneko and S. Nakao \cite{Kaneko-Nakao}, the
pathwise uniqueness of (\ref{e4}) is crucial for the convergence of the strong solution of (\ref{e3})
to that of (\ref{e4}) as $\ee \rightarrow 0$. Lemma \ref{lem5} is applied later for the convergence of the derivative flow
equation (\ref{sde1}).  We first need the following lemma on the convergence of stochastic integrals, which is essentially due to
A. V. Skorohod \cite{Skorohod},
see also I. Gy\"ongy and T. Martinez \cite[Lemma 5.2]{Gyongy-Martinez}.

\begin{lemma}(\cite{Skorohod})\label{lem4}
Let $W_t$ and $\{W_t^{(n)}\}_{n=1}^{\infty}$ be $\R^m$-valued Brownian motions, 
let $\xi(t)$ and  $\{\xi_n(t)\}_{n=1}^{\infty}$ be
$\R^{m \times d}$-valued stochastic processes such that for all $t \ge 0$ the following It\^o integrals are well defined:
 $$I_n(t):=\int_0^t \xi_n(s)dW_s^{(n)}, \quad I(t):=\int_0^t \xi(s)dW_s.$$
  Suppose that for some $T>0$, $\lim_{n\to \infty}
  \sup_{t \in [0,T]}|\xi_n(t)-\xi(t)|=0$ and
$\lim_{n\to \infty} \sup_{t \in [0,T]}|W^{(n)}_t-W_t|=0$ with convergence in probability.
Assume that for some $\delta>0$,
\begin{equation}\label{lem4-1}
\sup_{n}\int_0^T \E \left(|\xi_n(t)|^{2+\delta}\right)dt<\infty.
\end{equation}
 Then for every $\kappa>0$,
\begin{equation*}
\lim_{n \to \infty}\mathbb{P}\Big(\sup_{t \in [0,T]}|I_n(t)-I(t)|\ge \kappa\Big)=0.
\end{equation*}
\end{lemma}

\begin{proof} Let $R>0$. Define
$ \xi_n^R(t):=\big(\xi_n(t)\wedge R\big)\vee (-R)$,
$ \xi^R(t):=\big(\xi(t)\wedge R\big)\vee (-R)$ and
\begin{equation*}
\begin{split}
I_n^R(t):=\int_0^t \xi_n^R(s)dW_s^{(n)},\quad
I^R(t):=\int_0^t \xi^R(s)dW_s,
\end{split}
\end{equation*}
where $a \wedge b:=\min(a,b)$, $a \vee b:=\max(a,b)$ for every $a, b \in \R$.
Since the stochastic proceses $\{(\xi_n^R(t), t\in [0,T]), n\in \mathbb{N}_+\}$ and
$\{\xi^R(t), t\in[0,T]\}$ are uniformly bounded
and $\xi_n^R(t) \to \xi^R(t)$ in probability as $n \to \infty$,
we may apply Lemma 5.2 in  I. Gy\"ongy-T. Martinez \cite{Gyongy-Martinez} to obtain
\begin{equation*}
\lim_{n \to \infty}\p\Big(\sup_{t \in [0,T]}|I_n^R(t)-I^R(t)|\ge \kappa\Big)=0.
\end{equation*}
By Burkholder-Davis-Gundy  inequality, Chebyshev inequality and H\"older inequality,
\begin{equation}\label{lem4-3}
\begin{split}
&\p\left(\sup_{t \in [0,T]}|I_n^R(t)-I_n(t)|\ge \kappa\right)\le
{1\over \kappa^2} \sup_n\E\left(\sup_{t \in [0,T]}|I_n^R(t)-I_n(t)|^2\right)\\
&\le{C\over \kappa^2}\sup_n \E\left(\int_0^T |\xi_n(s)|^2 1_{\{|\xi_n(s)|>R\}} ds\right)\\
&\le \frac{1}{\kappa^2R^\delta} \sup_n\int_0^T \E \left( |\xi_n(s)|^{2+\delta} \right)ds.
\end{split}
\end{equation}
By (\ref{lem4-1}) the above term converges to zero uniformly for $n$ as $R\to \infty$.

By taking a subsequence if necessary we know
$\lim_{n \to \infty}\sup_{t \in [0,T]}|\xi_n(t)-\xi(t)|=0$, $\p-$ a.s..
Therefore by Fatou lemma and (\ref{lem4-1}) we obtain
\begin{equation}\label{lem4-3a}
\begin{split}
& \int_0^T \E \left( |\xi(s)|^{2+\delta} \right)ds
\le \liminf_{n \to \infty}\int_0^T \E \left( |\xi_n(s)|^{2+\delta} \right)ds\\
&\le \sup_n\int_0^T \E \left( |\xi_n(s)|^{2+\delta} \right)ds<\infty.
\end{split}
\end{equation}
So based on (\ref{lem4-3a}) and following the same procedure in (\ref{lem4-3})
we have
\begin{equation*}
\lim_{R \to \infty}
\p\left(\sup_{t \in [0,T]}|I^R(t)-I(t)|\ge \kappa\right)=0.
\end{equation*}
Note that for every $R>0$,
\begin{equation*}
\begin{split}
& \p\Big(\sup_{t \in [0,T]}|I_n(t)-I(t)|\ge \kappa\Big)
\le \p\Big(\sup_{t \in [0,T]}|I^R(t)-I(t)|\ge \kappa\Big)\\
&+\p\Big(\sup_{t \in [0,T]}|I_n^R(t)-I_n(t)|\ge \kappa\Big)
+\p\Big(\sup_{t \in [0,T]}|I_n^R(t)-I^R(t)|\ge \kappa\Big),
\end{split}
\end{equation*}
we first take  $n\to \infty$ then take $R\to \infty$ to complete the proof.
\end{proof}

Following the proof in \cite[Theorem A]{Kaneko-Nakao} and  \cite[Theorem 2.2]{Gyongy-Martinez}, we can show the
following result about the convergence of general SDE (\ref{e3}), which is suitable for our application
(to the derivative flow equation).

\begin{lemma}\label{lem5}
Fix a $T>0$, let $\mu^{\ee,x}$ denote the distribution of the process $(\phi^\epsilon_{\cdot}(x),$ $ t\le T)$
on the  path space $\W:=C([0,T];\R^d)$. Assume that pathwise uniqueness holds for (\ref{e4}).
We suppose that there exist some $p>2$ and $q>1$ such that  the following  conditions hold for  every compact set $K\subseteq \R^d$.
  \begin{enumerate}
 \item[(1)] For all $1\le k \le m$,
 \begin{equation}\label{lem5-1}
\begin{split}
&\sup_{\ee, \tilde \ee \in (0, \ee_0)}\sup_{x \in K}\int_0^T
\E\left( |Y_k^{\ee}(\phi^{\tilde \ee}_t(x))|^p\right)dt<\infty,\\
& \sup_{\ee, \tilde \ee \in (0, \ee_0)}\sup_{x \in K}\int_0^T
\E\left( |Y_0^{\ee}(\phi^{\tilde \ee}_t(x))|^q\right)dt<\infty;
\end{split}
\end{equation}

\item[(2)] For all $1\le k \le m$,
\begin{equation}\label{lem5-2}
\begin{split}
&\limsup_{\ee, \tilde \ee \to 0}\sup_{x \in K}\int_0^T
\E\left(|Y_k^{\ee}(\phi^{\ee}_t(x))-Y_k^{\tilde \ee}(\phi^{\ee}_t(x))|^p\right)dt=0,\\
&\limsup_{\ee, \tilde \ee \to 0}\sup_{x \in K}\int_0^T
\E\left(|Y_0^{\ee}(\phi^{\ee}_t(x))-Y_0^{\tilde \ee}(\phi^{\ee}_t(x))|^q\right)dt=0;
\end{split}
\end{equation}
\item[(3)] Let $\{x_n\}_{n=1}^{\infty} \subseteq K$ and $\{\ee_n\}_{n=1}^{\infty}\subseteq (0,\ee_0)$.
 If $\mu^{\ee_n,x_n}$ converges weakly to a limit measure $\mu^0$, then for every
$1\le k \le m$,
\begin{equation}\label{lem5-3}
\begin{split}
&\lim_{\ee \to 0}\int_0^T
\int_{\W}\left|Y_k^{\ee}(\sigma_t)-Y_k(\sigma_t)\right|^p \, \mu^0(d \sigma)\,dt=0,\\
&\lim_{\ee \to 0}\int_0^T
\int_{\W}\left|Y_0^{\ee}(\sigma_t)-Y_0(\sigma_t)\right|^q \, \mu^0(d \sigma)\,dt=0.
\end{split}
\end{equation}
 \end{enumerate}
Then for every $x\in \R^d$ there exists a unique complete strong solution $\phi_t(x)$ with initial point
$x \in \R^d$,  to (\ref{e4}). Moreover
for every compact set $K \subseteq \R^d$,
\begin{equation}\label{lem5-4}
\lim_{\ee \to 0} \sup_{x \in K}\E \left(\sup_{t \in [0,T]}|\phi_t^{\ee} (x)-\phi_t(x)|\right)=0.
\end{equation}
\end{lemma}

\begin{proof}
We suppose that there is a compact set $K_0 \subseteq \R^d$,
such that
\begin{equation}\label{lem5-5a}
\limsup_{\ee, \tilde \ee \to 0}
\sup_{x \in K_0}\E \left(\sup_{t \in [0,T]} |\phi_t^{\ee} (x)-\phi_t^{\tilde \ee}(x)|\right)>0,
\end{equation}
then there exist $\kappa>0$, $\{x_n\}_{n=1}^{\infty}\subseteq K_0$, and  two sequences $\{\ee_{n,1}\}_{n=1}^{\infty}$,
$\{\ee_{n,2}\}_{n=1}^{\infty}$ contained in $ (0,\ee_0)$ such that
\begin{equation}\label{lem5-5}\lim_{n\to \infty}\E \left(\sup_{t \in [0,T]}
|\phi_t^{\ee_{n,1}}(x_n) -\phi_t^{\ee_{n,2}}(x_n)|\right)>\kappa.
\end{equation}
Let $z^n_{\cdot}=\big(\phi_{\cdot}^{\ee_{n,1}}(x_n), \phi_{\cdot}^{\ee_{n,2}}(x_n), W_{\cdot} \big)$ and
let $\nu^{n}$ be the distribution of
$z^n_{\cdot}$ on the  path space $C([0,T];\R^{2d+m})$.

Note that $z_t^{n}$ is a semi-martingale, we will
apply \cite[Theorem 3]{Zheng85} or \cite{Rebolledo} to show that
the family of probability measures $\{\nu^{n}\}_{n=1}^{\infty}$
on $C([0,T]; \R^{2d+m})$ is tight.
In particular, as
in \cite[Theorem 3]{Zheng85} or \cite{Rebolledo}, it suffices to verify the uniformly bounded property for the
variational processes and the drift processes of the semi-martingales $\{z^n_{\cdot}\}_{n=1}^{\infty}$.

Note that $\phi_{t}^{\ee_{n,i}}(x_n)=x_n+M_t^{n,i}+A_t^{n,i}$, $i=1,2$, where
\begin{equation*}
\begin{split}
&M_t^{n,i}:=\sum_{k=1}^m\int_0^t Y_k^{\ee_{n,i}}
\big(\phi_{s}^{\ee_{n,i}}(x_n)\big)dW_s^k
,\ \ A_t^{n,i}:=\int_0^t Y_0^{\ee_{n,i}}
\big(\phi_{s}^{\ee_{n,i}}(x_n)\big)ds.
\end{split}
\end{equation*}
Let $\left\<M^{n,i} \right\>_t$ be the variational process for
$M^{n,i}$. We define
\begin{equation*}
\begin{split}
& u^{n,i}_t:=\sum_{k=1}^m|Y_k^{\ee_{n,i}}
\big(\phi_{t}^{\ee_{n,i}}(x_n)\big)|^2,
\ \  a^{n,i}_t:=Y_0^{\ee_{n,i}}
\big(\phi_{t}^{\ee_{n,i}}(x_n)\big).
\end{split}
\end{equation*}
Hence
\begin{equation*}
\begin{split}
& \left\<M^{n,i} \right\>_t=\int_0^t u^{n,i}_s ds,
\ \  A^{n,i}_t=
\int_0^t a^{n,i}_s ds.
\end{split}
\end{equation*}
From (\ref{lem5-1}) we know for $p':=\min\{\frac{p}{2},q\}>1$,
\begin{equation*}
\sup_n\E\Big(\int_0^T |u^{n,i}_t|^{p'}dt\Big)<\infty,
\ \sup_n\E\Big(\int_0^T |a^{n,i}_t|^{p'}dt\Big)<\infty,\ i=1,2,
\end{equation*}
which implies that the following random variables
\begin{equation*}
\Big\{x_n,\int_0^T |u^{n,i}_t|^{p'}dt, \int_0^T |a^{n,i}_t|^{p'}dt,\
n \in \mathbb{N}_+,\ i=1,2\Big\}
\end{equation*}
are uniformly bounded in probability. Therefore according to
\cite[Theorem 3]{Zheng85}, $\{\nu^{n}\}_{n=1}^{\infty}$ is tight.

By the Skorohod theorem, see e.g. Theorem 2.7 of Chapter 1 in \cite{Ikeda-Watanabe},
we can find a subsequence of $\{z_{\cdot}^n\}_{n=1}^{\infty}$ which will also be denoted by
$\{z_{\cdot}^n\}_{n=1}^{\infty}$ for simplicity, and
there exists a probability space $(\tilde \Omega, \tilde \F, \tilde \p)$ on which
there is
a sequence of $\R^{2d+m}$-valued stochastic processes $\tilde z_{\cdot}^n:=\big(\tilde y_\cdot^{n,1},
\tilde y_\cdot^{n,2}, \tilde W_\cdot^n\big)$ with the property that $\tilde z_{\cdot}^n$ has the same
distribution with $z^n_{\cdot}$, and
\begin{equation}\label{lem5-6}
\lim_{n \to \infty}\sup_{t \in [0,T]}\big|\tilde z_t^n-\tilde z_t\big|=0,\ \ \ \
\tilde \p-a.s.
\end{equation}
for some $\R^{2d+m}$-valued process $\tilde z_{\cdot}=\big(\tilde y_\cdot^{1},\tilde y_\cdot^{2}, \tilde W_\cdot\big)$.

Condition (\ref{lem5-1}) implies that
 $\{\sup_{t \in [0,T]}|\tilde z_t^n|\}_{n=1}^{\infty}$ is uniformly integrable which follows from
a  round of BDG inequality and H\"older inequality, therefore we have
\begin{equation*}
\lim_{n \to \infty}\tilde \E \left(\sup_{t \in [0,T]}\big|\tilde z_t^n-\tilde z_t\big|\right)=0.
\end{equation*}
By  (\ref{lem5-5}) we also obtain that
\begin{equation}\label{lem5-8}
\tilde \E\left(\sup_{t \in [0,T]}\big|\tilde y_t^1-\tilde y_t^2\big|\right)>\kappa.
\end{equation}

Since $\tilde z_\cdot^n\stackrel {law}{=}z_\cdot^n$, for every $0\le s<t\le T$, $\tilde W_t^n-\tilde W_s^n$ is independent of the $\sigma$-algebra
$\G_s^n:=\sigma\{ \tilde z_{r}^n;\ 0\le r\le s\}$. Hence for every $j\in \mathbb{N}_+$ and
$f \in C_b(\R^{(2d+m)j})$, $g \in C_b(\R^m)$,
$0< s_1< s_2< \dots < s_j<s<t\le T$,
\begin{equation*}
\begin{split}
& \tilde \E \left(g(\tilde W_t^n-\tilde W_s^n)f(\tilde z_{s_1}^n, \dots, \tilde z_{s_j}^n)\right)
=\tilde \E \left(g(\tilde W_t^n-\tilde W_s^n)\right)\tilde \E
\left(f(\tilde z_{s_1}^n, \dots, \tilde z_{s_j}^n)\right).
\end{split}
\end{equation*}
Set $\G_s=\sigma\{\tilde z_{r};\ 0\le r\le s\}$. Taking $n\to \infty$ in the above identity and using (\ref{lem5-6}) we obtain that
\begin{equation*}
\begin{split}
& \tilde \E \left(g(\tilde W_t-\tilde W_s)f(\tilde z_{s_1}, \dots, \tilde z_{s_j})\right)
=\tilde \E \left(g(\tilde W_t-\tilde W_s)\right)\tilde \E
\left(f(\tilde z_{s_1}, \dots, \tilde z_{s_j})\right),
\end{split}
\end{equation*}
which implies that $\tilde W_t-\tilde W_s$ is independent of the $\sigma$-algebra $\G_s$. Since $\tilde W_{\cdot}$ is the limit of  the family of Brownian motions $\tilde W_\cdot^{n}$, it has the same finite dimensional distribution as $W_{\cdot}$,
therefore $\tilde W_{\cdot}$ is a Brownian motion with respect to the filtration $(\G_s, 0\le s\le T)$.

In the computation below we will drop the index $1$, so $\tilde y_t^{n,1}$, $\tilde y_t^1$,  $\ee_{n,1}$
 will be denoted by $\tilde y_t^{n}$, $\tilde y_t$
and  $\ee_{n}$ respectively.
We  use again the fact that $\tilde z_\cdot^n\stackrel {law}{=}z_\cdot^n$ to observe that $(\tilde y_t^{n}, \tilde W_t^n)$
is a strong solution to
SDE (\ref{e3}) with
$\ee=\ee_{n}$, i.e.
\begin{equation}\label{lem5-9}
\tilde y_t^{n}=x_n+\sum_{k=1}^m \int_0^t Y_k^{\ee_{n}}(\tilde y_s^{n})\, d\tilde W_s^{n,k}+
\int_0^t Y_0^{\ee_{n}}(\tilde y_s^{n})\, ds,
\end{equation}
where $\tilde W_t^n=(\tilde W_t^{n,1},\dots,\tilde W_t^{n,m})$ denotes the components of
$\tilde W_t^n$. Next we will take the limit  $n\to \infty$ in (\ref{lem5-9}) to prove that
 $(\tilde y_t, \tilde W_t)$ is a strong solution to (\ref{e4}).


For  a fixed $n_0\in \mathbb N_+$, we define
\begin{equation*}
\begin{split}
& I_1^{n,n_0}(t)=\sum_{k=1}^m\int_0^t \left(Y_k^{\ee_{n}}(\tilde y_s^{n})-Y_k^{\ee_{n_0}}(\tilde y_s^{n})\right)d\tilde W_s^{n,k}\\
& I_2^{n,n_0}(t)=\sum_{k=1}^m\left(\int_0^t Y_k^{\ee_{n_0}}(\tilde y_s^{n})\,d\tilde W_s^{n,k}-
\int_0^t Y_k^{\ee_{n_0}}(\tilde y_s)\, d\tilde W_s^{k}\right),\\
& I_3^{n_0}(t)=\sum_{k=1}^m\int_0^t  \left( Y_k^{\ee_{n_0}}(\tilde y_s)-Y_k(\tilde y_s)\right)\,d\tilde W_s^{k}.
\end{split}
\end{equation*}

We use condition (\ref{lem5-2}), BDG inequality and H\"older inequality to obtain the following estimate for $I_1^{n,n_0}$,
\begin{equation}\label{lem5-7}
\begin{split}
&\limsup_{n_0 \to \infty}\limsup_{n \to \infty}\tilde \E\left(
\sup_{t \in [0,T]}|I_1^{n,n_0}(t)|^p\right)\\
&\le C(p)\sum_{k=1}^m\limsup_{n_0 \to \infty}\limsup_{n \to \infty}
\tilde \E\left(\int_0^T | Y_k^{\ee_n}
(\tilde y_t^n)-Y_k^{\ee_{n_0}}
(\tilde y_t^n)|^2 dt\right)^{\frac{p}{2}} \\
&\le C(p,T) \sum_{k=1}^{m}
\limsup_{\ee, \tilde \ee \to 0}\sup_{x \in K_0}\int_0^T
\E\big(|Y_k^{\ee}(\phi^{\ee}_t(x))-Y_k^{\tilde \ee}(\phi^{\ee}_t(x))|^p\big)\, dt=0.
\end{split}
\end{equation}
Now we work on the second integral. Since $Y_k^{\ee_{n_0}} \in C^{\infty}(\R^d;\R^d)$,
by (\ref{lem5-6}), we know for every
fixed $n_0$,
\begin{equation*}
\lim_{n \to \infty}\sup_{t \in [0,T]}
|Y_k^{\ee_{n_0}}(\tilde y_t^n)-Y_k^{\ee_{n_0}}(\tilde y_t)|=0,\
\ \tilde \p-a.s..
\end{equation*}
Due to
condition (\ref{lem5-1}), we may apply the convergence  Lemma \ref{lem4} for stochastic integrals
and conclude that  for every fixed
$n_0$, $\sup_{t \in [0,T]}|I_2^{n,n_0}(t)|$ converges to $0$ in probability as
$n \to \infty$. In an analogous way to (\ref{lem5-7}), by condition (\ref{lem5-1}),
we can show that $\{\sup_{t \in [0,T]}|I_2^{n,n_0}(t)|^2\}_{n=1}^{\infty}$  is uniformly integrable,
therefore for every fixed $n_0$,
\begin{equation*}
\limsup_{n \to \infty}
\tilde \E\left( \sup_{t \in [0,T]}|I_2^{n,n_0}(t)|^2\right)=0.
\end{equation*}
From (\ref{lem5-6}) the distribution $\mu^0$ of $\tilde y_{\cdot}$ is a weak limit
of  $\mu^{\ee_n,x_n}$,  therefore the condition (\ref{lem5-3})  can be applied to the third integral
and we have
\begin{equation*}
\limsup_{n_0 \to \infty}
\tilde \E\left(\sup_{t \in [0,T]}|I_3^{n_0}(t)|^2\right)=0.
\end{equation*}
Combing all the estimates above for $I_1^{n,n_0}$, $I_2^{n,n_0}$,
$I_3^{n_0}$,  we first take $n \to \infty$ then take $n_0 \to \infty$ to obtain
\begin{equation*}
\begin{split}
& \lim_{n \to \infty}
\sum_{k=1}^m \tilde \E\left( \sup_{t \in [0,T]}\left|\int_0^t Y_k^{\ee_{n}}(\tilde y_s^{n})\,d\tilde W_s^{n,k}
-\int_0^t Y_k(\tilde y_s)\,d\tilde W_s^{k}\right|^2\right)=0.
\end{split}
\end{equation*}
By the same method we also prove that
\begin{equation*}
\begin{split}
& \lim_{n \to \infty}
\tilde \E\left(\sup_{t \in [0,T]}\left|\int_0^t Y_0^{\ee_{n}}(\tilde y_s^{n})ds
-\int_0^t Y_0(\tilde y_s)ds\right|\right)=0.
\end{split}
\end{equation*}
Finally we take $n \to \infty$ in (\ref{lem5-9}) to see that
\begin{equation*}
\tilde y_t=x_0+\sum_{k=1}^m \int_0^t Y_k(\tilde y_s)d\tilde W_s^{k}+
\int_0^t Y_0(\tilde y_s)ds.
\end{equation*}
The above argument applies equally to $\tilde y_t^{n,2}$ and we prove that both  $(\tilde y_t^1, \tilde W_t)$
and $(\tilde y_t^2, \tilde W_t)$  are $\G_t$ adapted strong solution to (\ref{e4}) with initial value $x_0$.
Consequently by the pathwise uniqueness for (\ref{e4}), for every
$t \in [0,T]$, $\tilde y_t^1=\tilde y_t^2$, $\tilde \p-a.s.$,  and
\begin{equation*}
\tilde \E\left(\sup_{t \in [0,T]}|\tilde y_t^1-\tilde y_t^2|\right)=0,
\end{equation*}
which contradicts with (\ref{lem5-8}). So the assumption (\ref{lem5-5a}) is not true, the sequence
$\sup_{t \in [0,T]} |\phi_t^{\ee} (x)$ $-\phi_t^{\tilde \ee}(x)|$ must be a Cauchy sequence
as $\ee, \tilde \ee \to 0$, and there exists
a stochastic process $\phi_{\cdot}(x)$, such that the convergence in (\ref{lem5-4}) holds.  By the same
approximation argument above, $(\phi_{\cdot}(x), W_\cdot)$ is the unique complete strong solution to
(\ref{e4}) with initial point $x$.
\end{proof}



\section{An estimate for the probability distribution }\label{sec-distribution}
Let
$\L={1\over 2} \sum_{i,j=1}^d$ $ a_{i,j}$ ${\partial^2\over \partial x_i\partial x_j}$
$ +\sum_{i=1}^d X_{0i}{\partial \over \partial x_i} $.
If $A(x)$ is strictly elliptic,
 $\{X_k\}_{k=0}^m$ are bounded and uniformly H\"older continuous,  there is a Gaussian type
 upper and lower bound  for the
 fundamental solution to the parabolic PDE ${\partial u_t\over \partial t}=\L u_t$. Such estimates are used in our earlier work  \cite{Chen-Li-Arxiv}, an unpublished notes.
But under Assumption \ref{assumption1}, we are not sure whether  such estimate is true, so we will apply
Lemma \ref{lem3} instead.

We first cite a lemma on the distributions of continuous semi-martingales,
which is a special case of that in N. V. Krylov \cite[Lemma 5.1]{Krylov-86}, see also
I. Gy\"ongy and T. Martinez \cite[Lemma 3.1]{Gyongy-Martinez}. Let $\det (A)$ and ${\rm tr}(A)$ denote respectively  the determinant and the trace of a
$d \times d$ matrix $A$.

\begin{lemma}(\cite{Krylov-86})\label{lem1}
Suppose that $F_t(x)$ is a strong solution to (\ref{sde}) with initial point $x \in \R^d$, set
$\tilde F_t(x):=F_t(x)-x$.
For every $q\ge d+1$, $T>0$, $R>0$ and Borel measurable
function $f: \R_+ \times \R^d \rightarrow \R_+$, letting $\tau_R(x):=\inf\{t \ge 0, \ |\tilde F_t(x)|>R\}$,
we have
\begin{equation}\label{lem1-0}
\begin{split}
&\E\left(\int_0^{T \wedge \tau_R(x)}  f(t,\tilde F_t(x))\big( {\det A(F_t(x))}\big)^{\frac{1}{q}}\,dt\right)\\
&\le C(d)  \,e^{T} \left(\mathbf{A}(R)+\mathbf{B}(R)^2\right)^{\frac{d}{2q}} \left(\int_0^T\int_{|x|\le R}f^q(t,x)dx \; dt\right)^{\frac{1}{q}},
\end{split}
\end{equation}
where  $C(d)$ is a constant  depending only on
$d$ and
\begin{equation}\label{lem1-1}
\begin{split}
&\mathbf{A}(R)=\E\left(\int_0^{T \wedge \tau_R(x)} {\rm tr} A(F_t(x))dt\right),\\
&\mathbf{B}(R)=\E\left(\int_0^{T \wedge \tau_R(x)} |X_0(F_t(x))|dt\right).
\end{split}
\end{equation}
\end{lemma}
\begin{proof}
In  \cite[Lemma 3.1]{{Gyongy-Martinez}}, we take $X(t)=\tilde F_t(x)$, $A(t)=t$,
$d \,m(t)=X(F_t(x))\,dW_t$, $dB(t)=X_0(F_t(x))dt$, $\gamma=T$, $r(t)=1$,
$c(t)=1_{[0,T]}(t)$, $p=q-1$, and the conclusion follows.
\end{proof}

We also cite the following lemma, \cite[Lemma 6.1]{Li-flow}, which is concerned with the moment estimates for
(\ref{sde}), the regularity condition imposed on $\{X_k\}$ in \cite{Li-flow} will not be needed.
\begin{lemma}(\cite{Li-flow})\label{lem2}
Suppose that $\{X_k\}_{k=0}^{\infty}$ are locally bounded vector fields. 
Let $g: \R^d \rightarrow \R_+$ be a positive $C^2$ function. For any $\lambda>0$ let
\begin{equation}\label{e1}
\Theta_g(\lambda) =\sup_{x\in \R^d}\Big\{ \big(D g(X_0)\big)(x)+\frac{1}{2}\sum_{k=1}^m\left(  \lambda|Dg(X_k)|^2+ D^2 g(X_k,X_k)\right)(x)\Big\}.
\end{equation}
If furthermore $ \Theta_g(\lambda)<\infty $,  then for every $t>0$ and stopping time $\tau<\zeta(x)$,  we have
\begin{equation*}
\E(\e^{\lambda g(F_{t\wedge \tau}(x))})\le \e^{\lambda (g(x)+\Theta_g(\lambda)t)},
\end{equation*}
where  $F_t(x)$ is a  strong solution to (\ref{sde}) with  initial point $x \in \R^d$, and $\zeta(x)$ is the explosion time
of $F_t(x)$.
\end{lemma}
\begin{proof}
The conclusion is just that of \cite[Lemma 6.1]{Li-flow}.
In \cite[Lemma 6.1]{Li-flow} the coefficients are assumed to be
$C^1$, by carefully tracking the proof, we observe that the regularity condition
in \cite[Lemma 6.1]{Li-flow} is not needed.

In fact, it  suffices to show the case where $\lambda=1$. Since
$F_t(x)$ is a strong solution to (\ref{sde}), by definition it is also
a semi-martingale. By It\^o formula,
we have for each $t>0$ and stopping time $\tau<\zeta(x)$,
\begin{equation*}
g(F_{t \wedge \tau}(x))=g(x)+N_{t\wedge \tau}-\frac{\langle N \rangle_{t \wedge \tau}}{2}+b_{t \wedge \tau},
\end{equation*}
where
\begin{equation*}
\begin{split}
 N_t=&\int_0^t Dg(F_s(x))(X(F_s(x)))dW_s, \\
 b_t=&\int_0^t \Big(\frac{1}{2}\sum_{k=1}^m |Dg(F_s(x))(X_k(F_s(x)))|^2+D g(F_s(x))(X_0(F_s(x)))\Big)ds\\
&+
\frac{1}{2}\sum_{k=1}^m \int_0^t D^2 g(F_s(x))(X_k(F_s(x)),X_k(F_s(x)))ds,
\end{split}
\end{equation*}
and  $\langle N \rangle_t$ denotes the variational process of $N_t$.
By the definition of $\Theta_g(1)$,  $b_t\le  t\Theta_g(1)$ and  we have,
\begin{equation*}
\exp\left( g(F_{t\wedge \tau_R \wedge \tau}(x))\right)\le \exp\big(g(x)+\Theta_g(1)t\big)\exp\left(
N_{t\wedge \tau_R \wedge \tau}-\frac{1}{2}\langle N \rangle_{t\wedge \tau_R\wedge \tau}\right),
\end{equation*}
where $\tau_R:=\inf\{t\ge 0;\ |F_t(x)-x|>R\}$.
Since $\{X_k\}_{k=0}^{\infty}$ are locally bounded, 
$\exp\big( N_{t\wedge \tau_R \wedge \tau}-
\frac{1}{2}\langle N \rangle_{t\wedge \tau_R\wedge \tau}\big)$ is a martingale for each $R>0$, we take expectations of both sides of
the inequality above and let $R \rightarrow \infty$,
then the required conclusion follows from Fatou's lemma.
\end{proof}



\begin{example}\label{ex4-1}
Suppose that $\{X_k\}_{k=0}^{\infty}$ are locally bounded vector fields. Assume that for every $p>0$ there is $C(p)>0$ such that,
\begin{equation}\label{ex4-1-1}
\sum_{k=1}^m p|X_k(x)|^2+\langle x, X_0(x)\rangle\le C(p)(1+|x|^2).
\end{equation}
We apply Lemma \ref{lem2} to $g(x)=\log(1+|x|^2)$. Since \begin{equation*}
\Theta_g(\lambda)
\le C(\lambda)\sup_{x \in \R^d} {1\over 1+|x|^2} \left (p(\lambda)
\sum_{k=1}^m |X_k(x)|^2+\langle x, X_0(x)\rangle\right)<\infty,
\end{equation*}
 we have for every $p>0$, $R>0$,
 \begin{equation*}
\E\big(|F_{t \wedge \tau_R}(x)|^{p}\big) \le C(p) \e^{C(p)t}(|x|^{p}+1)
\end{equation*}
for some constant $C(p)>0$ independent of $R$. Therefore let $R \to \infty$, we obtain,
\begin{equation}\label{e2}
\E\big(|F_t(x)|^{p}\big) \le C(p) \e^{C(p)t}(|x|^{p}+1).
\end{equation}
In particular, SDE (\ref{sde}) is complete if (\ref{ex4-1-1}) holds.
\end{example}


\begin{lemma}\label{lem3}
Let $F_t(x)$ be a  strong solution to (\ref{sde}) with  initial value $x \in \R^d$. Suppose
that the conditions (\ref{c1}), (\ref{c2aa}) and (\ref{ex4-1-1}) hold. Then for every
$p > d+1$, $T>0$ and non-negative measurable function
$f: \R_+ \times \R^d \rightarrow \R_+$, we have
\begin{equation}\label{lem3-1}
\begin{split}
\E \left(\int_0^T f(t,F_t(x))\,dt\right)
\le Q_1(T)Q_2(x) \left(\int_0^T \int_{\R^d} f^{p}(t,y)dy dt \right)^{1\over p},
\end{split}
\end{equation}
where $Q_1:\R_+ \to \R_+$, $Q_2: \R^d \rightarrow \R_+$ are positive Borel measurable functions which
only depend on $d$, $p$ and the constants in (\ref{c1}), (\ref{c2aa}) and (\ref{ex4-1-1}), such that
$\sup_{T \in [0,\tilde T_0]}Q_1(T)<\infty$ and $\sup_{x \in K}Q_2(x)<\infty$
for every $\tilde T_0>0$ and compact set $K \subseteq \R^d$.
\end{lemma}

\begin{proof}
 Let $\tilde f(t,y):=f(t,y+x),\ y \in \R^d$, $\tilde F_t(x):=F_t(x)-x$ and $\alpha:=\frac{p}{ d+1}$.
 Note that from Example \ref{ex4-1}, we know the solution
$F_t(x)$ is non-explode, then applying
H\"older's inequality with exponent $\alpha>1$ and
Lemma \ref{lem1} with $q=d+1$,
and letting $R \to \infty$ in (\ref{lem1-0}), by Fatou lemma
we have
\begin{equation}\label{lem3-2}
\begin{split}
&\E\left( \int_0^{T } f(t,F_t(x))dt\right)=
\E\left( \int_0^{T } \tilde f(t,\tilde F_t(x))dt\right)\\
&\le  \left(\E \left(\int_0^{T} \left(\det A(F_t(x))\right)^{\frac{1}{d+1}}
\tilde f^{\alpha}(t,\tilde F_t(x))dt\right) \right)^{\frac{1}{\alpha}}\\
&\cdot \left(\E \left (\int_0^{T }\big(\det A(F_t(x))
\big)^{-\frac{1}{(d+1)(\alpha-1)}}dt\right)\right)^{\frac{\alpha-1}{\alpha}}\\
&\le (C(d)e^T) ^{1\over \alpha} \sup_{R>0}\big(\mathbf{A}(R)+\mathbf{B}(R)^2\big)^{\frac{d}{2(d+1)\alpha}}
\left( \int_0^T\int_{\R^d} |f|^{p}(t,y) dy dt\right)^{1\over p}\\
&\cdot
\left(\E \left (\int_0^{T }\big(\det A(F_t(x))
\big)^{-\frac{1}{(d+1)(\alpha-1)}}dt\right)\right)^{\frac{\alpha-1}{\alpha}},
\end{split}
\end{equation}
where we use the translation invariant property for the Lebesgue integral, i.e.
$\int_{\R^d} |\tilde f(t,y)|^p$ $ dy=\int_{\R^d} |f(t,y)|^p dy$, and
the constant $\mathbf{A}(R)$, $\mathbf{B}(R)$
are defined by (\ref{lem1-1}).

Since (\ref{ex4-1-1}) holds, by Example \ref{ex4-1} we know that the moment estimate (\ref{e2}) is true. From (\ref{c2aa})
we have the following estimate,
\begin{equation}\label{lem3-3}
\begin{split}
&\sup_{R>0}\mathbf{B}(R)\le
\E\left(\int_0^{T} |X_0(F_t(x))| dt\right) \\
&\le C\E\left(\int_0^{T} (1+|F_t(x)|^{p_2})dt\right)\le  C \e^{CT}T(1+|x|^{p_2}).
\end{split}
\end{equation}
For ${\rm tr}(A)={\rm tr} (X^*X)$, we apply again (\ref{c2aa}) and (\ref{e2}) to obtain
 \begin{equation}\label{lem3-4}
\begin{split}
\sup_{R>0}\mathbf{A}(R) \le \E\left(\int_0^{T} \text{tr} A(F_t(x))dt\right)\le
C\e^{CT}T(1+|x|^{2p_2}).
\end{split}
\end{equation}
Similarly, by the ellipticity condition (\ref{c1}), ${\det (A(x))}^{-\frac{1}{(d+1)(\alpha-1)}}$ $\le
C(1+|x|^{\frac{d p_1}{(d+1)(\alpha-1)}})$, and we have,
 \begin{equation}\label{lem3-5}
\begin{split}
\E \left (\int_0^{T }\big(\det A(F_t(x))
\big)^{-\frac{1}{(d+1)(\alpha-1)}}dt\right)
\le C\e^{CT}T(1+|x|^{\frac{d p_1}{(d+1)(\alpha-1)}}).
\end{split}
\end{equation}
In particular, it is easy to check that all the constants $C$ above only depend on
the constants in (\ref{c1}), (\ref{c2aa}) and (\ref{ex4-1-1}).

Putting the estimates (\ref{lem3-3})-(\ref{lem3-5}) into (\ref{lem3-2}), we can show
(\ref{lem3-1}) with
\begin{equation*}
\begin{split}Q_1(T)&=\e^{C(1+T)}T^{\alpha-1\over \alpha}(T+T^2)^{d\over 2(d+1)\alpha}
,\ Q_2(x)= 1+|x|^{d(p_1+p_2)\over (d+1)\alpha}.
\end{split}
\end{equation*}
\end{proof}

\section{Construction of the approximation vector fields}\label{sec-construction}

We will construct a class of approximation SDEs with smooth and elliptic coefficients for (\ref{sde1}).
Let $\eta:\R^d\rightarrow \R_{+}$  be the smooth mollifier defined by
$\eta(x)=C\e^{\frac{1}{|x|^2-1}}\1_{\{ |x|<1\}}$,
where $C$ is a normalizing constant such that $\int_{\R^d} \eta(x)dx=1$.
For every $\ee>0$, set $\eta_{\varepsilon}(x):=\varepsilon^{-d}\eta(\frac{x}{\varepsilon})$.
For $f\in L^1_{\loc}(\R^d)$, we let $f*\eta_{\ee}$ denote the  convolution of $f$ with $\eta_{\ee}$,
\begin{equation*}
f*\eta_{\ee}(x):=\int_{\R^d}\eta_\ee(x-y)f(y)dy=\int_{|y-x|\le \ee}\eta_\ee(x-y)f(y)dy,\ \ \ x \in \R^d.
\end{equation*}

It is natural to approximate each $X_k$ by $C^\infty$ smooth vector field $X_k*\eta_\ee$.
However, since we do not make the assumption  that $X_k$  are bounded, the approximating systems
$\{X_k*\eta_\ee\}_{k=1}^m$ may loose ellipticity if $\ee$ is small enough.

Suppose that Assumption \ref{assumption1} holds, in particular, the condition that
$X_k \in W_{loc}^{1,p_3}(\R^d;\R^d)$, $1 \le k$ $ \le m$ and
$X_0 \in W_{loc}^{1,p_4}(\R^d;\R^d)$ for some constants $p_3>2(d+1)$, $p_4>d+1$ ensures
that $X_k$, $0 \le k \le m$ are continuous.
Then for every  $R\ge R_1+1$ we may define the truncated vector field $\tilde X_{k,R}$ as following,
\begin{equation}\label{lem6-1}
\tilde X_{k,R}\big((\rho,\theta)\big):=
\begin{cases}
X_k(0),&\text{if}\; \rho=0,\\
X_k\big((\rho,\theta)\big), \quad &\text{if}\  0<|\rho|\leqslant R, \\
            X_k\big((R,\theta)\big), &\text{if}\ |\rho|> R,
            
\end{cases}
\end{equation}
where $R_1$ is the constant in Assumption \ref{assumption1} (4),
$(\rho,\theta) \in \R_+\times \mathbb{S}^{d-1}$ denotes the spherical
coordinate in $\R^d$, $0$ denotes the origin of $\R^d$.
We first state a
technical lemma for $\{\tilde X_{k,R}\}_{k=0}^{\infty}$.
\begin{lemma}\label{lem6}
If  Assumption \ref{assumption1}  holds for $\{X_k\}_{k=0}^m$, then for every
$R $ $\ge R_1+1$,  Assumption \ref{assumption1}
holds for $\{\tilde X_{k,R}\}_{k=0}^m$ with the corresponding constants independent of $R$.
\end{lemma}
\begin{proof}
Since the right hand side of (\ref{c1}-\ref{c2aa}) depends only on $|x|$,
 it is clear from the definition (\ref{lem6-1}) that they hold true for
$\{\tilde X_{k,R}\}_{k=0}^m$ with the same constants $C_1$, $C_2$.

Suppose that (\ref{c2}) holds with a constant $0<\delta\le 1$. 
For every $x,y \in \R^d$ such that $|y|\le \frac{\delta}{2}$,
 if $|x+y|\le R$,  then $\tilde X_{k,R}(x+y)=X_k(x+y)$ by definition,
 so according to (\ref{c2}) we have
\begin{equation}\label{lem6-0aa}
\begin{split}
&\sup_{\{y \in \R^d;\ |y|\le\frac{\delta}{2},|y+x|\le R\}}
\left(p \sum_{k=1}^m|\tilde X_{k,R}(x+y)|^2+ \langle x, \tilde X_{0,R}(x+y)\rangle\right)\\
&\le \sup_{|z|\le \delta}
\left(p \sum_{k=1}^m|X_{k}(x+z)|^2+ \langle x, X_{0}(x+z)\rangle\right)
\le C(p)(1+|x|^2).
\end{split}
\end{equation}
Let
\begin{equation*}
\begin{split}
& B_R=\{x \in \R^d;\ |x|<R\},
\ \ \ \ \ S_R=\{x \in \R^d; \ |x|=R\}.
\end{split}
\end{equation*}
For every $z \in \R^d$ such that $z \notin 0$, we denote the spherical coordinate of
$z$ by $(|z|, \theta(z))$ with $\theta(z) \in \mathbb{S}^{d-1}$. And for every
$z \in \R^d$,
we define $\pi_R:\R^d\to S_R$ to be the the shortest distance projection, i.e., $\pi_R(z):=(R, \theta(z))$.

If $|x+y|>R$, then by definition $\tilde X_{k,R}(x+y)=$ $X_k(\pi_R(x+y))
=X_k\big((R, \theta(x+y))\big)$, and we obtain
\begin{equation}\label{lem6-1a}
\begin{split}
&
p \sum_{k=1}^m|\tilde X_{k,R}(x+y)|^2+ \langle x, \tilde X_{0,R}(x+y)\rangle\\
&=p
\sum_{k=1}^m \left|  X_{k}\Big(\pi_R(x+y)\Big)\right|^2+{|x|\over R}
\Big \langle \pi_R(x), X_0(\pi_R(x+y))\Big \rangle.
\end{split}
\end{equation}
For  $\theta_1, \theta_2 \in \mathbb S^{d-1}$ we define
$\tilde g(\theta_1, \theta_2):=\big \langle (1,\theta_1),  (1,\theta_2)\big \rangle $,
to be the Euclidean inner product of the corresponding points in $S_1$.
Hence  for every $z_1, z_2 \in \R^d$ such that $z_1, z_2 \notin 0$, $\<z_1,z_2\>=$
$\<(|z_1|, \theta(z_1)), (|z_2|, $ $\theta(z_2))\>$ $=|z_1| |z_2| \tilde g(\theta(z_1),$ $ \theta(z_2))$.
When $|x+y| \ge R>2$ and $|y|\le \frac{\delta}{2}$, then $|x|>R-1$ and we have
\begin{equation}\label{lem6-0a}
\begin{split}
&\frac{\delta^2}{4} \ge |x+y-x|^2=|x|^2+|x+y|^2-2\<x+y,x\>\\
& =|x|^2+|x+y|^2-2|x||x+y|\tilde g\Big(\theta(x), \theta(x+y)\Big)\\
& \ge 2(R-1)^2-2(R-1)^2\tilde g\Big(\theta(x), \theta(x+y)\Big) \\
&=\Big|\frac{R-1}{R}\Big|^2\left|(R,\theta(x))-(R,\theta(x+y))\right|^2 \ge \frac{1}{4}\left|\pi_R(x)-\pi_R(x+y)\right|^2,
\end{split}
\end{equation}
where the second inequality above holds since $|x+y|>R-1$, $|x|>R-1$ and $\Big|g\Big(\theta(x), \theta(x+y)\Big)\Big|\le 1$.
(\ref{lem6-0a}) proves that   $\big|\pi_R(x)-\pi_R(x+y)\Big|\le \delta$ for every $|y|\le {\delta \over 2}$.
To the right hand of   (\ref{lem6-1a}) we apply (\ref{c2}) for the system  $\{X_k\}_{k=0}^m$ at the point $\pi_R(x)$ to obtain that
\begin{equation*}\label{lem6-0}
\begin{split}
&\sup_{\{y \in \R^d;\ |y|\le\frac{\delta}{2},|y+x|> R\}} \left(p \sum_{k=1}^m|\tilde X_{k,R}(x+y)|^2
+ \langle x, \tilde X_{0,R}(x+y)\rangle\right)\\
&=
{|x|\over R} \sup_{\{y \in \R^d;\ |y|\le\frac{\delta}{2},|y+x|> R\}} \Bigg(
p\frac{R}{|x|}
\sum_{k=1}^m \left|  X_{k}\Big(\pi_R(x+y)\Big)\right|^2\\
&+
\Big \langle \pi_R(x), X_0(\pi_R(x+y))\Big \rangle\Bigg)\\
&\le  {|x|\over R} \sup_{|z|\le \delta}   \left(2p
\sum_{k=1}^m \left|  X_{k}\Big(\pi_R(x)+z\Big)\right|^2+
\Big \langle \pi_R(x), X_0(\pi_R(x)+z)\Big \rangle\right)\\
&\le {C(2p)|x|\over R} (1+R^2)
\le C C(2p)(1+|x|^2).
\end{split}
\end{equation*}
Note that $|x|>R-1$, here the first inequality is due to the property ${|x|\over R}\ge \frac{1}{2}$
and the last step is due to the property  $\frac{1+R^2}{R}\ge C(1+|x|)$.
Together with (\ref{lem6-0aa}), this shows that (\ref{c2}) holds with the corresponding constant $\delta$ replaced by $\frac{\delta}{2}$.


Now we move on to item (3) of Assumption \ref{assumption1} and prove first that there is a constant
 $p_3>2(d+1)$, such that
$X_{k,R} \in W_{\loc}^{1,p_3}(\R^d; \R^d)$ for $1\le k \le m$.

Since $\tilde X_{k,R}(x)=X_k(x)$ for every $|x|\le R$,
$\tilde X_{k,R} \in W^{1,p_3}(B_R;\R^d)$ by Assumption \ref{assumption1} (3), also note that
the boundary $\partial B_R$ is $C^1$ and $X_k$ is continuous,
we may apply the integration by parts formula  to $\tilde X_{k,R}$, therefore for
every $\psi \in C_0^{\infty}(\R^d)$ and $1\le i \le d$,
\begin{equation}\label{lem6-2}
\int_{B_R} D_i X_{k} (x)\psi (x)dx=-\int_{B_R} \tilde X_{k,R} (x) D_i \psi (x)dx+
\int_{S_R} \tilde X_{k,R} \psi \nu_i dS,
\end{equation}
where $D_i \psi=\partial_{x_i}\psi$, $\nu=(\nu_1, \dots \nu_d)$ denotes the outward normal vector field on
$S_R$,  $dS$ denotes integration with respect to the area measure on $S_R$.

By (\ref{c4}), $DX_k$ is locally bounded on the complement $B_{R_1}^c$ of $B_{R_1}$, hence
$X_k $ is  locally Lipschitz continuous on $B_{R_1}^c$ and belongs to $ W_{\loc}^{1,\infty}(B_{R_1}^c;\R^d)$,
see e.g.  \cite[Theorem 4  in Section 5.8.2]{Evans}).
For every $x=(|x|,\theta(x))$ and $y=(|y|,\theta(y))$  with   $ R\le |x| \le |y| $
and $\theta(x), \theta(y) \in \mathbb{S}^{d-1}$, we have,
\begin{equation*}
\begin{split}
& |\tilde X_{k,R}(x)-\tilde X_{k,R}(y)|= \big|X_k\big(\pi_R(x)\big)-X_k\big(\pi_R(y)\big)\big|\\
&\le C_3 (1+|R|^{p_5})\big|\pi_R(x)-\pi_R(y)\big|
\le CC_3(1+|R|^{p_5})|x-y|
\end{split}
\end{equation*}
where first inequality is due to the Lipschitz continuity of $X_k$ on
$S_R$ and (\ref{c4}), and the second inequality is by (\ref{lem6-0a}).
We conclude that the truncated vector field $\tilde X_{k,R}$ is globally Lipschitz continuous on $B_R^c$, and
$\tilde X_{k,R} \in W^{1,\infty}(B_R^c; $ $ \R^d)$.
Applying again integration by parts formula to $\tilde X_{k,R}$, for every $\psi \in C_0^{\infty}(\R^d)$,
\begin{equation}\label{lem6-3}
\int_{B_R^c} D_i\tilde X_{k,R} (x)\psi(x) dx
=-\int_{B_R^c} \tilde X_{k,R} (x)D_i \psi (x)dx
-\int_{S_R} \tilde X_{k,R} \psi \nu_i dS,
\end{equation}
where we use the property that the outward normal vector on
$\partial B_R^c$ is $-\nu$.  From  (\ref{lem6-2}) and  (\ref{lem6-3}) we see that
for every $\psi \in C_0^{\infty}(\R^d)$ and $1\le i \le d$,
\begin{equation*}
\begin{split}
&\int_{\R^d} \big(D_iX_k(x)\1_{\{x \in B_R\}}+ D_i \tilde X_{k,R}(x)\1_{\{x \in B_R^c\}}\big) \psi(x)
dx\\
&=-\int_{\R^d} \tilde X_{k,R} (x)D_i \psi (x)dx,
\end{split}
\end{equation*}
which means that $\tilde X_{k,R}$ is weakly differentiable with the differential $D\tilde X_{k,R}$, and
$D\tilde X_{k,R}(x)=DX_k(x)\1_{\{x \in B_R\}}+ D \tilde X_{k,R}(x)\1_{\{x \in B_R^c\}}$, then we conclude
$\tilde X_{k,R} \in W^{1,p_3}_{\loc}(\R^d;\R^d)$ from the fact that $X_k \in W^{1,p_3}(B_R;\R^d)$ and
$\tilde X_{k,R} \in W^{1,\infty}(B_R^c;\R^d)$. As the same way we can show
$\tilde X_{0,R} \in W^{1,p_4}_{\loc}(\R^d;\R^d)$.

Let $\nu(\theta)$  be the unit outward normal vector of $S_R$ at the point $(R, \theta)$,
by the definition  of $\tilde X_{k,R}$,
for almost every $x=(|x|,\theta(x))\in \R^d$ with $|x|>R$ we obtain,
\begin{equation}\label{lem6-4a}
 D\tilde X_{k,R}(x)\big(\nu(\theta(x))\big)=0. \ \
\end{equation}
Let $T_{\theta}S_R$ be the tangent space to the sphere $S_R$ at the point $(R, \theta)$, since
$\tilde X_{k,R}$ is Lipschitz continuous on $S_R$, by Rademacher's theorem the derivative
$D\tilde X_{k,R}$ in the directions of $T_{\theta}S_R$  is almost everywhere well defined with respect to the area measure
on $S_R$.
For every $\xi \in T_{\theta(x)}S_{|x|}$, by a standard isomorphism, we can also assume
$\xi \in T_{\theta(x)}S_{R}$. And  by definition (\ref{lem6-1}), for
almost every $x=(|x|,\theta(x))\in \R^d$ with $|x|>R$ and every
$\xi \in T_{\theta(x)}S_{|x|}$,
\begin{equation}\label{lem6-4}
D\tilde X_{k,R}(x)(\xi)=\frac{R}{|x|}DX_{k}\big(\pi_R(x)\big)(\xi).
\end{equation}


For every $p>1$, let $\tilde K_{p,R}(x):=\sup_{|\xi|=1}\tilde H_{p,R}(x)\big(\xi,\xi\big)$, where
\begin{equation}\label{lem6-5a}
\begin{split}
\tilde H_{p,R}(x)\big(\xi,\xi\big)&=2p\langle D \tilde X_{0,R}(x)(\xi), \xi\rangle
+(2p-1)p\sum_{k=1}^{m}|D \tilde X_{k,R}(x)(\xi)|^2.\end{split}
\end{equation}


By (\ref{lem6-4a}), for  almost every $x=(|x|,\theta(x)) \in \R^d$ with $|x|>R\ge R_1+1$,
$\tilde H_{p,R}(x)\big(\nu(\theta(x)),$ $\nu(\theta(x))\big)=0$, so by
(\ref{lem6-4}) and (\ref{c4aa}), we have
\begin{equation*}
\begin{split}
& \tilde K_{p,R}(x)=\max\Big\{0,
\sup_{\xi \in T_{\theta(x)}S_{|x|}, |\xi|=1} H_p\big(\pi_R(x)\big)
\big(\xi,\xi\big) \Big\}  \le  0 \vee  K_p\big(\pi_R(x)\big)\\
&\le C(p)\log (1+|R|^2) \le C(p)\log (1+|x|^2).
\end{split}
\end{equation*}
On the other hand, it is obvious that $\tilde K_{p,R}(x)=K_p(x)$ for almost every
$x \in \R^d$ with $|x|<R$.
So we obtain that (\ref{c4aa}) holds for $\tilde X_{k,R}$ with the same constants $C(p)$ and $R_1$ as that for
$X_k$.

So for the constant $\kappa(p)$ in (\ref{c3}) and every $\tilde R>0$,
\begin{equation}\label{lem6-5}
\begin{split}
& \sup_{R>R_1}\int_{\{|x|\le \tilde R\}}e^{\kappa(p)\tilde K_{p,R}(x)}dx\\
&\le \sup_{R>R_1} \Big(\int_{\{|x|\le  R_1\}}e^{\kappa(p)K_p(x)}dx+\int_{\{R_1<|x|\le  R\}}e^{\kappa(p)K_p(x)}dx \\
&+\int_{\{R< |x|\le \tilde R\}}e^{ \kappa(p)\left(0 \vee K_p\big(\pi_R(x)\big)\right)}dx\Big)\\
&\le \int_{\{|x|\le  R_1\}}e^{\kappa(p)K_p(x)}+
\int_{\{R_1< |x|\le \tilde R\}}e^{\kappa(p)C(p)\log(1+|x|^2)} dx
<\infty,
\end{split}
\end{equation}
which means (\ref{c3}) is true for $\{\tilde X_{k,R}\}_{k=0}^m$ with the corresponding constants
independent of $R$.
Similarly, we can show (\ref{c4}) holds for $\{\tilde X_{k,R}\}_{k=0}^m$ with the corresponding constants independent of $R$.
\end{proof}

For every $\ee>0$ we define the approximating vector fields $\{X_k^{\ee}\}_{k=0}^m$ by
$X_k^{\ee}:=\tilde X_{k,\ee^{-\lambda}}*\eta_{\ee}$, where the constant $\lambda>0$ will be chosen  later in Lemma \ref{lem7}.
Since for every $\ee>0$, $\tilde X_{k,\ee^{-\lambda}}$ is bounded, it is obvious that
$X_k^{\ee} \in C_b^{\infty}(\R^d;\R^d)$.  Following result concerns about the properties of
$\{X_k^{\ee}\}_{k=0}^m$ which are uniformly for $\ee$.


\begin{lemma}\label{lem7}
Suppose Assumption \ref{assumption1} holds. There exist $\lambda_0>0$, $\ee_0>0$, such that
if we define $X_k^{\ee}:=\tilde X_{k,\ee^{-\lambda_0}}*\eta_{\ee}$,
then for every
$\ee \in (0,\ee_0)$, (\ref{c1}), (\ref{c2aa}), (\ref{c3})-(\ref{c4aa}) hold for
$\{X_k^\ee\}_{k=0}^m$ with the corresponding constants independent of $\ee$. Furthermore,
for every $p>0$, there exists a $C(p)>0$, such that
\begin{equation}\label{lem7-0}
\sup_{\ee \in (0,\ee_0)}\Big(\sum_{k=1}^m p|X_k^\ee (x)|^2+\<x, X_0^\ee(x)\>\Big)\le C(p)(1+|x|^2).
\end{equation}
\end{lemma}

\begin{proof}
In the proof we fix a $\lambda>0$ which will be determined later, we set
$\ee_1(\lambda):=\min((R_1+2)^{-\frac{1}{\lambda}}, {\delta \over 4})$.


Since
$\ee_1^{-\lambda}\ge R_1+2$, from  Lemma \ref{lem6} we have,
\begin{equation*}
\sup_{\ee \in (0,\ee_1)}\sup_{|y|\le \frac{\delta}{2}}\Big(p\sum_{k=1}^m |\tilde X_{k,\ee^{-\lambda}}(x+y)|^2
+\langle x,  \tilde X_{0,\ee^{-\lambda}}(x+y)\rangle \Big)
\le C(p)(1+|x|^2).
\end{equation*}
For every  $\ee <\ee_1<\frac{\delta}{2}$, we apply this to $X_{k}^{\ee}(x)=\int_{|y| \le \ee}
\tilde X_{k,\ee^{-\lambda}}(x-y)\eta_{\ee}(y)dy$
and by Jensen's inequality we obtain
\begin{equation}\label{lem7-0a}
\begin{split}
&\sup_{\ee \in (0,\ee_1)}\Big(p\sum_{k=1}^m |X_{k}^\ee(x)|^2
+\langle x,  X_{0}^\ee(x)\rangle \Big)\\
&\le \sup_{\ee \in (0,\ee_1)}\Big(p\sum_{k=1}^m
\int_{|y| \le \frac{\delta}{2}} |\tilde X_{k,\ee^{-\lambda}}(x-y)|^2\eta_{\ee}(y)dy\\
&+
\int_{|y| \le \frac{\delta}{2}} \big\langle x, \tilde X_{0,\ee^{-\lambda}}(x-y)\big\rangle\eta_{\ee}(y)dy \Big)
\le C(p)(1+|x|^2),
\end{split}
\end{equation}
which means (\ref{lem7-0}) holds. Similarly, we can show (\ref{c2aa}) holds for
$\{X_k^{\ee}\}_{k=0}^m$ with
the corresponding constants independent of $\ee$.

Let $K_p^{\ee}(x):=\sup_{|\xi|=1}H_p^{\ee}(x)\big(\xi,\xi\big)$ where
\begin{equation}\label{lem7-1a}
H_{p}^{\ee}(x)\big(\xi,\xi\big):=2p\langle D X_0^{\ee}(x)(\xi), \xi\rangle
+(2p-1)p\sum_{k=1}^{m}|D X_k^{\ee}(x)(\xi)|^2.
\end{equation}
The local integrability (\ref{c3}) is trivial for the smooth functions
$X_k^\epsilon$. Now we try to give an uniform bounds for $\ee$. As the same argument for
(\ref{lem7-0a}), according to Jensen's inequality
we have $K^{\ee}_p \le \tilde K_{p,\ee^{-\lambda}}$ $*\eta_{\ee}$, where
$\tilde K_{p,\ee^{-\lambda}}$ is defined by (\ref{lem6-5a}).
Letting $\kappa(p)$ be the constant in (\ref{c3}),  by Jensen's inequality and (\ref{lem6-5}), for every $p>1$,
$R>0$,
\begin{equation}\label{lem7-1b}
\begin{split}
&\sup_{\ee \in (0,\ee_1)}
\int_{\{|x|\le  R\}}\exp\big(\kappa(p) K_p^{\ee}(x)\big)dx\\
&\le
\sup_{\ee \in (0,\ee_1)}
\int_{\{|x|\le  R\}}\exp\Big(\kappa(p) \tilde K_{p,\ee^{-\lambda}}*\eta_\ee(x)\Big)dx\\
&\le \sup_{\ee \in (0,\ee_1)}
\int_{\{|x|\le R\}}\Big(\exp\big(\kappa(p) \tilde K_{p,\ee^{-\lambda}}\big)*\eta_{\ee}(x)\Big)dx\\
&\le
 \sup_{\ee \in (0,\ee_1)}\int_{\{|x|\le  R+1\}} \exp \left(\kappa(p) \tilde K_{p, \ee^{-\lambda}}(x)\right)\;dx<\infty.
\end{split}
\end{equation}
Hence (\ref{c3}) holds
for $\{X_k^\ee\}_{k=0}^m$ with the corresponding constants independent of $\ee$.
As the similar way, we can check (\ref{c4}) and (\ref{c4aa}) hold for
$\{X_k^\ee\}_{k=0}^m$ with the corresponding constants independent of $\ee$.

Finally we study the ellipticity condition (\ref{c1}).
By (\ref{lem6-4a}) and (\ref{lem6-4}), for every $\ee \in (0,\ee_1)$, $1 \le k \le m$,
\begin{equation*}
\begin{split}
& \sup_{|y|\ge R_1}|D\tilde X_{k,\ee^{-\lambda}}(y)| \le
\sup_{R_1 \le |y|\le \ee^{-\lambda}}|DX_k(y)|\le C(1+\ee^{-\lambda p_5}).
\end{split}
\end{equation*}
Therefore we have,
 \begin{equation}\label{lem7-2a}
\begin{split}
&\big|\tilde X_{k,\ee^{-\lambda}}(x)-\tilde X_{k,\ee^{-\lambda}}(y)\big|\le
C(1+\ee^{-\lambda p_5})|x-y|,\ \ x,y \in B_{R_1}^c.
\end{split}
\end{equation}
On the other hand, by (\ref{lem6-1}), for every $\ee \in (0,\ee_1)$ and $x \in \R^d$ with $|x|\le R_1+2\le \ee^{-\lambda}$,
we know that  $\tilde X_{k,\ee^{-\lambda}} (x)=X_k(x)$. Since $X_k \in W^{1,p_3}_{\loc}(\R^d;\R^d)$ for some constant
$p_3>2(d+1)$, according to the Sobolev embedding lemma we have,
\begin{equation}\label{lem7-3a}
\begin{split}
\sup_{\ee \in (0,\ee_1)}\big|\tilde X_{k,\ee^{-\lambda}}(x)-\tilde X_{k,\ee^{-\lambda}}(y)\big|
\le C|x-y|^{\iota}, \ x,y \in B_{R_1+2}\\
\end{split}
\end{equation}
for some constant $\iota \in (0,1)$, which is independent of $\ee$.
Then by (\ref{lem7-2a}) and (\ref{lem7-3a}), for every $\ee \in (0,\ee_1)$,
\begin{equation}\label{lem7-2}
\begin{split}
&\big|X_k^{\ee}(x)-X_{k,\ee^{-\lambda}}(x)\big|\
\le \int_{|y|\le \ee} \big|\tilde X_{k,\ee^{-\lambda}}(x+y)-\tilde X_{k,\ee^{-\lambda}}(x)\big|
\eta_{\ee}(y)dy\\
&\le C\ee^{\iota}\1_{\{|x|\le R_1+1\}}+C(1+\ee^{-\lambda p_5})\ee \1_{\{|x|>R_1+1\}}.
\end{split}
\end{equation}
We write the components of $X_k^\ee$ as $X_k^\ee=(X_{k1}^\ee,\cdots, X_{kd}^\ee)$ and
for every $1\le i,$ $j \le d$ we define
\begin{equation*}
\begin{split}
a_{i,j}^\ee(x):=\sum_{k=1}^m X_{ki}^\ee(x)
X_{kj}^\ee(x),\ \ \tilde a_{i,j}^{\ee}(x):=\sum_{k=1}^m \tilde X_{ki,\ee^{-\lambda}}(x)
\tilde X_{kj,\ee^{-\lambda}}(x).
\end{split}
\end{equation*}
By (\ref{c2aa}) and definition (\ref{lem6-1}), for every $\ee \in (0,\ee_1)$ and $x \in \R^d$,
\begin{equation*}
\begin{split}
&|\tilde X_{k,\ee^{-\lambda}}(x)|\le \sup_{|x|\le \ee^{-\lambda}}|X_k(x)|\le C(1+\ee^{-\lambda p_2}),
\end{split}
\end{equation*}
therefore we have
\begin{equation*}
\begin{split}
&|X_k^{\ee}(x)|\le \int_{|y|\le \ee}|\tilde X_{k,\ee^{-\lambda}}(x+y)|\eta_{\ee}(y)dy\le  C(1+\ee^{-\lambda p_2}).
\end{split}
\end{equation*}
Combing this with (\ref{lem7-2}) we get
\begin{equation}\label{lem7-4}
\begin{split}
& |a_{i,j}^{\ee}(x)-\tilde a_{i,j}^{\ee}(x)|\\
&\le C\sup_{1\le k \le m}
\big|X_{k}^{\ee}(x)-\tilde X_{k,\ee^{-\lambda}}(x)\big|
\big(|X_{k}^{\ee}(x)|+|\tilde X_{k,\ee^{-\lambda}}(x)|\big)\\
&\le C\ee^{\iota-\lambda p_2}\1_{\{|x|\le R_1+1\}}+C
\ee^{1-\lambda(p_2+p_5)}\1_{\{|x|>R_1+1\}}.
\end{split}
\end{equation}
By definition (\ref{lem6-1}), and ellipticity condition (\ref{c1}),
for every $\ee \in (0,\ee_1(\lambda))$ and $\xi=(\xi_1,\cdots, \xi_d) \in \R^d$ with $|\xi|=1$,
\begin{equation}\label{lem7-4a}
\sum_{i,j=1}^d\tilde a_{i,j}^{\ee}(x)\xi_i\xi_j
\ge \frac{C}{1+|x|^{p_1}}\1_{\{|x|\le \ee^{-\lambda}\}}+{C \over 2}\ee^{\lambda p_1}\1_{\{|x|> \ee^{-\lambda}\}}.
\end{equation}
We will prove below that the error made by convolution does not affect the ellipticity of $\{a_{i,j}^{\ee}\}$.
In fact, according to (\ref{lem7-4}) and (\ref{lem7-4a}),
\begin{equation}\label{lem7-3}
\begin{split}
& \sum_{i,j=1}^d a_{i,j}^{\ee}(x)\xi_i\xi_j\\
&\ge \sum_{i,j=1}^d \tilde a_{i,j}^{\ee}(x)\xi_i\xi_j-d^2\max_i |\xi_i|^2
\sup_{\ee \in (0,\ee_1)}\sup_{i,j}|a_{i,j}^{\ee}(x)-\tilde a_{i,j}^{\ee}(x)|\\
&\ge C\left(\frac{\1_{\{|x|\le \ee^{-\lambda}\}}}{1+|x|^{p_1}}-\ee^{\iota-\lambda p_2}
\1_{\{|x|\le R_1+1\}}-\ee^{1-\lambda(p_2+p_5)}\1_{\{R_1+1<|x|\le \ee^{-\lambda}\}}\right)\\
&+C\left(\ee^{\lambda p_1}-\ee^{1-\lambda(p_2+p_5)}
\right)\1_{\{|x|> \ee^{-\lambda}\}}.
\end{split}
\end{equation}

We choose a constant $\lambda_0>0$ small enough satisfying
$\lambda_0 p_1< \iota-\lambda_0 p_2$ and $\lambda_0 p_1<1-\lambda_0(p_2+p_5)$.
Hence for such $\lambda_0$, there exists a positive constant $\ee_0(\lambda_0)<\ee_1(\lambda_0)$, such that for
every $\ee \in (0,\ee_0)$,
\begin{equation*}
\begin{split}
&\ee^{\iota-\lambda_0 p_2} \le \frac{\ee^{\lambda_0 p_1}}{4 (1+\ee^{\lambda_0 p_1})},\ \ \
\ee^{1-\lambda_0(p_2+p_5)} \le \frac{\ee^{\lambda_0 p_1}}{4(1+\ee^{\lambda_0 p_1})}
\le \frac{\ee^{\lambda_0 p_1}}{4}.
\end{split}
\end{equation*}
So for every $\ee \in (0,\ee_0)$, $x \in \R^d$ with $|x|\le \ee^{-\lambda_0}$,
\begin{equation*}
\begin{split}
&\frac{ 1}{1+|x|^{p_1}}-\ee^{\kappa-\lambda_0 p_2}
\1_{\{|x|\le R_1+1\}}-\ee^{1-\lambda_0(p_2+p_5)}\1_{\{R_1+1<|x|\le \ee^{-\lambda_0}\}}\\
&\ge \frac{ 1}{1+|x|^{p_1}}- \frac{\ee^{\lambda_0 p_1}}{2(1+\ee^{\lambda_0 p_1})}\ge
\frac{ 1}{2(1+|x|^{p_1})}.
\end{split}
\end{equation*}
Now we fix the constant $\lambda_0$ and $\ee_0(\lambda_0)$ obtained above,
putting above estimates together into (\ref{lem7-3}), we have for every $\ee \in (0, \ee_0)$,
\begin{equation*}
\begin{split}
& \sum_{i,j=1}^d a_{i,j}^{\ee}(x)\xi_i\xi_j
\ge \frac{C}{2(1+|x|^{p_1})}\1_{\{|x|\le \ee^{-\lambda_0}\}}+
{C\over 2}|\ee|^{\lambda_0 p_1} \1_{\{|x|> \ee^{-\lambda_0}\}}\\
&\ge
\frac{C}{2} \left(\frac {1}{1+|x|^{p_1}}\right),
\end{split}
\end{equation*}
which means (\ref{c3}) holds for $\{X_k^\ee\}_{k=1}^m$, $\ee \in (0,\ee_0)$
with the corresponding constants independent of $\ee$.
\end{proof}

From now on  we take the constants $\lambda_0$ and $\ee_0$ to be that
obtained in Lemma \ref{lem7}, and for every $\ee \in (0,\ee_0)$, we define $X_k^\ee(x):=X_{k,\ee^{-\lambda_0}}* \eta_{\ee}$.
\begin{lemma}\label{lem8}
Suppose that Assumption \ref{assumption1} holds. For every $R>0$ and $p>1$,
\begin{equation}\label{lem8-1}
\begin{split}
& \lim_{\ee \to 0}\int_{\{|x|\le R\}}|X_k^{\ee}(x)-X_k(x)|^{p}dx=0,\ \ 0 \le k \le m,
 \end{split}
\end{equation}
\begin{equation}\label{lem8-2}
\lim_{\ee \to 0}\int_{\{|x|\le R\}}|DX_k^{\ee}(x)-DX_k(x)|^{p_3}dx=0,\ \ 1 \le k \le m,
\end{equation}
\begin{equation}\label{lem8-3}
\lim_{\ee \to 0}\int_{\{|x|\le R\}}|DX_0^{\ee}(x)-DX_0(x)|^{p_4}dx=0,
\end{equation}
where $p_3>2(d+1)$, $p_4>d+1$ are the constants in (3) of Assumption \ref{assumption1}.
\end{lemma}
\begin{proof}
For every fixed $R>0$ and every $\ee$ small enough such that
$\ee^{-\lambda}>R+1$, by definition (\ref{lem6-1}) we have $\tilde X_{k,\ee^{-\lambda}}(x)=X_k(x)$
for all $x \in \R^d$ with $|x|\le R+1$. Therefore for every $x \in \R^d$ with $|x|\le R$,
\begin{equation*}
DX_k^{\ee}(x)=D \tilde X_{k,\ee^{-\lambda}}*\eta_{\ee}(x)=D X_{k}*\eta_{\ee}(x),
\end{equation*}
Hence (\ref{lem8-2}) holds since $X_k \in W^{1,p_3}_{\loc}(\R^d;\R^d)$, $1\le k \le m$. As the same way we can show
(\ref{lem8-3}).

Since the $\{X_k\}_{k=0}^m$ are locally bounded by part (2) of Assumption \ref{assumption1}, similarly we can prove
(\ref{lem8-1}) for any $p>1$.


\end{proof}

\section{The derivative flow equation}\label{sec-derivative-flow}

Through this section, let $X_k^{\ee} \in C_b^{\infty}(\R^d;\R^d)$, $0 \le k \le m$, $\ee \in (0,\ee_0)$
be the vector fields constructed in Lemma \ref{lem7}, we consider the following approximating
SDE for (\ref{sde1}),
 \begin{equation}
 \begin{cases}
\label{approximated-sde}
&dx_t^\ee=\sum_{k=1}^m   X_k^\ee(x_t^\ee)dW_t^k+ X_0^\ee(x_t^\ee)dt,\\
&dv_t^\ee=\sum_{k=1}^m D X_k^\ee(x_t^\ee)(v_t^\ee)dW_t^k
+D X_0^\ee(x_t^\ee)(v_t^\ee)dt.
\end{cases}
\end{equation}
We denote the strong solution to (\ref{approximated-sde}) with initial point $(x,v)\in \R^{2d}$ by
$(F_t^\ee(x),$ $ V_t^\ee(x,v))$.

According to  Lemma \ref{lem7}, $\{X_k^\ee\}_{k=0}^m$ satisfies (\ref{c1}), (\ref{c2aa}) and (\ref{ex4-1-1}) with corresponding constants independent of $\ee$,
by a straightforward application of Lemma \ref{lem3} to $F_t^\ee(x)$, we obtain the following lemma, which will be frequently used in this section.


\begin{lemma}\label{lem3+lem7}
Suppose that Assumption (\ref{assumption1}) holds, then for
every
$p > d+1$, $T>0$, compact set $K \subseteq \R^d$,  and non-negative measurable function
$f: \R_+ \times \R^d \rightarrow \R$, we have,
\begin{equation}\label{lem3-1-2}
\begin{split}
\sup_{\ee \in (0,\ee_0)}\sup_{x \in K}\E \left(\int_0^T f(t,F_t^{\ee}(x))\,dt\right)
\le C(K)Q(T)\left(\int_0^T \int f^{p}(t,y)dy dt \right)^{1\over p},
\end{split}
\end{equation}
where $Q:\R_+ \to \R_+$ is a positive Borel measurable function such that
$\sup_{T \in [0,\tilde T_0]}$ $Q_1(T)<\infty$ for every
$\tilde T_0>0$ and $C(K)$ is a positive constant which may depend on $K$.
\end{lemma}

In this section, we will prove existence and uniqueness for (\ref{sde1}).
We first give the following lemma about the uniform moment estimate for $V_t^\ee(x,v)$.

\begin{lemma}\label{lem9}
Suppose that Assumption (\ref{assumption1}) holds.
Then for every $p\ge 2$ and compact set
$\tilde K \subseteq \R^{2d}$,
\begin{equation}\label{uniform-moments}
\sup_{\ee \in (0,\ee_0)}\sup_{(x,v)\in  \tilde K}\sup_{t \in [0,T_0(p)]} \E
\left(|V_t^\ee(x,v)|^p\right)<\infty,
\end{equation}
where $T_0(p):={\kappa(p)\over d+2}$ with the constant $\kappa(p)$ in (\ref{c3}).

\end{lemma}
\begin{proof}
Given $(x, v) \in  \R^{2d}$ fixed,
we write $(F_t^\ee,V_t^{\ee})$ for $(F_t^\ee(x),V_t^{\ee}(x,v))$ for simplicity.
We first follow some steps in \cite[Theorem 5.1]{Li-flow} (see also \cite{Li-moments}) for the estimation.
Applying It\^o formula to (\ref{approximated-sde}), we derive
\begin{equation}\label{lem9-1}
|V_t^\ee|^p =|v|^p+\sum_{k=1}^m \int_0^t |V_s^\ee|^p dM_s^{\ee}+ \int_0^t |V_s^\ee|^p da_s^{\ee},
\end{equation}
where
\begin{equation}\label{lem9-1a}
\begin{split}
& M_t^\ee:=p\sum_{k=1}^m \int_0^t \frac{\big\langle DX_k^\ee(F_s^\ee)
\big(V_s^\ee\big),V_s^\ee\big\rangle}
{|V_s^\ee|^2}dW_s^k,\ a_t^{\ee}:=\frac{p}{2}\int_0^t
\frac{\bar H_p^\ee(F_s^\ee)\big(V_s^\ee,V_s^\ee\big)}
{|V_s^\ee|^2}ds.
\end{split}
\end{equation}
Here for every $x\in \R^d, \xi \in \R^d$,
\begin{equation*}
\begin{split}
\bar H_p^\ee(x) \big(\xi,\xi\big)=& 2\langle DX_0^\ee(x)(\xi),\xi\rangle\\
&+\sum_{k=1}^m  \Big(|DX_k^\ee(x)(\xi)|^2
+ (p-2)
\frac{\big|\big\langle DX_k^\ee(x)
\big(\xi\big),\xi\big\rangle\big|^2}{|\xi|^2}\Big)
\end{split}
\end{equation*}
with the convention that $\frac{0}{0}=0$.

Furthermore, we know that for every $\R$-valued semi-martingale $N_t$, the unique solution to the linear equation
(in $\R$) $dz_t=z_t d N_t$ will have the expression $z_t=z_0\exp\big(N_t-\frac{\langle N \rangle_t}{2}\big)$, where
$\langle N \rangle_t$ denotes the quadratic variational process for $N_t$, see e.g.
\cite[Proposition 2.3 in Page 361]{Revuz-Yor} or \cite[Theorem 5.1]{Li-flow}. So by (\ref{lem9-1}) we have
\begin{equation}\label{lem9-2a}
|V_t^\ee|^p=|v|^p\exp\Big(M_t^\ee-\frac{\langle M^\ee\rangle_t}{2}+a_t^\ee\Big).
\end{equation}
Since $\tilde M_t^\ee:=\exp(2M_t^\ee-2\langle M^\ee \rangle_t)$ is a super martingale, $\E(\tilde M_t^\ee)\le 1$,
after applying H\"older inequality to (\ref{lem9-2a}) we deduce the following estimate,
\begin{equation}\label{lem9-2}
\begin{split}
 \E\left(|V_t^\ee|^p\right) &\le |v|^p\big(\E \tilde M_t^\ee\big)^{\frac{1}{2}}
\big(\E\left(\exp\big(\langle M^\ee \rangle_t+2a_t^\ee\big) \right)\big)^{\frac{1}{2}}\\
&\le |v|^p\left(\E\left(\exp\big( \int_0^t K_p^\ee(F_s^\ee)ds\big)\right)\right)^{\frac{1}{2}},
\end{split}
\end{equation}
where we use the property that
$\langle M^\ee \rangle_t+2a_t^\ee \le  \int_0^t  K_p^{\ee}(F_s^\ee)ds$
for $K_p^\ee$ defined by (\ref{lem7-1a}).
For every fixed $T>0$ and $t \in (0,T]$, by Jensen's inequality,
\begin{equation*}
\begin{split}
&\E\left(\exp\Big( \int_0^t K_p^\ee(F_s^\ee)ds\Big)\right)
=
\E\left(\exp\left( \int_0^T K_p^\ee(F_s^\ee)\1_{\{s \in (0,t)\}}ds\right)\right)\\
&\le \frac{1}{T}\left(
\E\left(\int_0^t \exp \left(T \; K_p^\ee(F_s^\ee) \right)ds\right)+(T-t)\right)\\
&\le  \frac{1}{T} \E\left(\int_0^t\exp\big(T\; K_p^\ee(F_s^\ee)\big)
\1_{\{|F_s^\ee|\le R_1+2\}}ds\right)\\
& + \frac{1}{T}  \E\left(\int_0^t\exp \left (T K_p^\ee(F_s^\ee)\right)\1_{\{|F_s^\ee|> R_1+2\}}ds\right)+1.
\end{split}
\end{equation*}
Applying Lemma \ref{lem3+lem7} with $p=d+2$, for every compact $K \subseteq \R^d$,
\begin{equation*}
\begin{split}
& \sup_{t \in [0,T]}\sup_{\ee \in (0,\ee_0)}\sup_{x \in K}
\E\left(\int_0^t\exp\Big(T K_p^\ee(F_s^\ee)\Big)
\1_{\{|F_s^\ee|\le R_1+2\}}\;ds\right)\\
&\le C(K,T)\sup_{\ee \in (0,\ee_0)}\left(\int_{\{|x|\le R_1+2\}}
\exp\Big(   T(d+2)\,K_p^\ee(x)\Big)\;dx\right)^{\frac{1}{d+2}}.
\end{split}
\end{equation*}
If   $T= T_0(p):={\kappa(p)\over d+2}$, the above quantity is finite by (\ref{lem7-1b}) in Lemma \ref{lem7}.

Also by Lemma \ref{lem7}, there is a constant $C(p)>0$ independent of $\ee$ such that
for every $x \in \R^d$ with $|x|>R_1+2$,
\begin{equation*}
\sup_{\ee \in (0,\ee_0)}K_p^\ee(x) \le C(p)\log(1+|x|^2).
\end{equation*}
By (\ref{lem7-0}) and Example \ref{ex4-1}, for every $p>0, T>0$ and compact set
$K \subseteq \R^d$,
\begin{equation}\label{lem9-3}
\sup_{\ee \in (0,\ee_0)}\sup_{x \in K}\sup_{t \in [0,T]}
\E\big(|F_t^\ee|^p\big)<\infty,
\end{equation}
therefore we have
\begin{equation*}
\begin{split}
& \sup_{\ee \in (0,\ee_0)}\sup_{x \in K}
\E\left(\int_0^{T_0(p)}\exp\big(T_0(p) K_p^\ee(F_s^\ee)\big)\1_{\{|F_s^\ee|> R_1+2\}}ds\right)\\
&\le \sup_{\ee \in (0,\ee_0)}\sup_{x_0 \in K}
\E\left(\int_0^{T_0(p)}\big(1+|F_s^\ee |^{2C(p)T_0(p)}\big)ds\right)<\infty.
\end{split}
\end{equation*}
We put all the estimates above back into (\ref{lem9-2}) to complete the proof.
\end{proof}

\begin{lemma}\label{lem10}
Suppose that Assumption \ref{assumption1} holds. Then for all $p>1$
and compact set $K \subseteq \R^d$,
\begin{equation}\label{lem10-1}
\begin{split}
&\limsup_{\ee, \tilde \ee \to 0}\sup_{x \in K}\int_0^T
\E\left(|X_k^{\ee}(F_t^\ee(x))-X_k^{\tilde \ee}(F_t^\ee(x))|^{p}\right)dt=0,
\ 0\le k \le m.
\end{split}
\end{equation}

Moreover, there exist constants $\beta_1>0$ and $\beta_2>0$
such that for all $1 \le k \le m$,  $T>0$, and compact subset $K \subseteq \R^d$, the following holds:

\begin{equation}\label{lem10-2}
\begin{split}
& \limsup_{\ee, \tilde \ee \to 0}\sup_{x \in K}\int_0^T
\E\left(|DX_k^{\ee}(F_t^\ee(x))-DX_k^{\tilde \ee}(F_t^\ee(x))|^{2+\beta_1}\right)dt=0,
\end{split}
\end{equation}
\begin{equation}\label{lem10-3}
\begin{split}
& \limsup_{\ee, \tilde \ee \to 0}\sup_{x \in K}\int_0^T
\E\left(|DX_0^{\ee}(F_t^\ee(x))-DX_0^{\tilde \ee}(F_t^\ee(x))|^{1+\beta_2}\right)dt=0.
\end{split}
\end{equation}
\begin{equation}\label{lem10-3a}
\begin{split}
& \sup_{\ee,\tilde \ee \in (0,\ee_0)}\sup_{x \in K}\E\left(\int_0^T \big
|DX_k^\ee\big(F_t^{\tilde \ee}(x)\big)\big|^{2+\beta_1}dt\right)<\infty,\\
& \sup_{\ee,\tilde \ee \in (0,\ee_0)}\sup_{x \in K}\E\left(\int_0^T \big
|DX_0^\ee\big(F_t^{\tilde \ee}(x)\big)\big|^{1+\beta_2}dt\right)<\infty.
\end{split}
\end{equation}
\end{lemma}
\begin{proof}
Given $x \in  \R^{d}$ fixed,
we write $F_t^\ee$ for $F_t^\ee(x)$ for simplicity.
We only prove (\ref{lem10-2}), the proof for (\ref{lem10-1}), (\ref{lem10-3}) and
(\ref{lem10-3a}) are similar.
 Let $p_3>2(d+1)$ be the constant in Assumption \ref{assumption1}(3). We take a $\delta_1$ $ \in (d+1,{p_3\over 2})$ and
 define $\beta_1:=\frac{p_3}{\delta_1}-2>0$. In particular, we have
$(2+\beta_1)\delta_1=p_3$.

Fix a $R>0$, we apply Lemma \ref{lem3+lem7}
to the function $\Big( |DX_k^{\ee}(F_t^\ee)-$ $DX_k^{\tilde \ee}(F_t^\ee)|^{2+\beta_1}$
$\1_{\{|F_t^\ee|\le R\}}\Big)$, and take  $p=\delta_1$ in  (\ref{lem3-1-2}) to obtain,
\begin{equation}\label{lem10-4}
\begin{split}
&\limsup_{\ee, \tilde \ee \to 0}\sup_{x \in K}\int_0^T\E\left( |DX_k^{\ee}(F_t^\ee)
-DX_k^{\tilde \ee}(F_t^\ee)|^{2+\beta_1}
\1_{\{|F_t^\ee|\le R\}}\right)dt \\
&\le \limsup_{\ee, \tilde \ee \to 0} \;C(K,T)\left(\int_{\{|x|\le R\}}
|DX_k^\ee(x)-DX_k^{\tilde \ee}(x)|^{p_3} \,dx \right)^{1 \over \delta_1}= 0.
\end{split}
\end{equation}
Here in the second step we also use Lemma \ref{lem8}.

By the statement of Lemma \ref{lem7},   (\ref{c4}) in Assumption \ref{assumption1}
 holds for every $\{X_k^\ee\}_{k=0}^m$ with the constants independent of $\ee$.
Thus for sufficiently large $R$ we have
\begin{equation*}
\sup_{\ee \in (0,\ee_0)}|DX_k^{\ee}(x)|\1_{\{|x|>R\}}\le C(1+|x|^{p_5})\1_{\{|x|>R\}}.
\end{equation*}
Then we obtain
\begin{equation}\label{lem10-5}
\begin{split}
&\sup_{\ee, \tilde \ee \in (0,\ee_0)}\sup_{x \in K}
\int_0^T\E\left(|DX_k^{\ee}(F_t^\ee)-DX_k^{\tilde \ee}(F_t^\ee)|^{2+\beta_1}\1_{\{|F_t^\ee|> R\}}\right)
dt\\
&\le 2 C\sup_{\ee\in (0,\ee_0)}\sup_{x \in K} \int_0^T \E\left(\left(1+|F_t^\ee|^{p_5(2+\beta_1)}
\right)\1_{\{|F_t^\ee|>R\}}\right) dt\\
&\le CR^{-p_5(2+\beta_1)} \sup_{\ee\in (0,\ee_0)}\sup_{x \in K} \int_0^T
\E \left(1+|F_t^\ee|^{2p_5(2+\beta_1)}\right)
dt\\
&\le  C(K,T)R^{-p_5(2+\beta_1)}.
\end{split}
\end{equation}
Here in the second step of inequality, we use H\"older inequality and Chebyshev inequality, and the third step is
due to the estimate (\ref{lem9-3}).

In  the inequalities (\ref{lem10-4}-\ref{lem10-5}) we first let $\ee, \tilde \ee \to 0$, then  let $R \to 0$, this gives
  (\ref{lem10-2}).
\end{proof}

We will show the pathwise uniqueness for the solution of (\ref{sde1}).

\begin{proposition}\label{prop1}
Under Assumption \ref{assumption1} pathwise uniqueness holds for the solution to (\ref{sde1}).
\end{proposition}

\begin{proof}
Given a  Brownian motion $W_t$, suppose $(x_t,v_t, W_t, \zeta)$ and $(\tilde x_t, \tilde v_t, W_t,\tilde \zeta)$ are
two strong solutions to (\ref{sde1}) with the same initial points, up to the explosion time $\zeta$, $\tilde \zeta$.
We already know that  Assumption \ref{assumption1} implies that any solution to
(\ref{sde}) is non-explode and the pathwise uniqueness holds for (\ref{sde}),
 see e.g. \cite[Theorem 1.3]{Zhang-11}, i.e.  $x_t=\tilde x_t$ $\p$-$a.s.$,  for every $t\ge 0$.
 Let $\bar v_t:=v_t-\tilde v_t$, it is easy to see that $\bar v_t$ satisfies the following linear
equation,
\begin{equation*}
\begin{split}
& d \bar v_t=\sum_{k=1}^m DX_k(x_t)(\bar v_t) dW_t^k+DX_0(x_t)(\bar v_t )dt,\ \ \bar v_0=0.
\end{split}
\end{equation*}
Since $DX_k\in L_{\loc}^{p_3}(\R^d;\R^d)$, $1\le k \le m$, $DX_0\in L_{\loc}^{p_4}(\R^d;$ $\R^d)$, and
by Assumption \ref{assumption1}, they have polynomial growth outside of $B_{R_1}$,
following the proof of Lemma \ref{lem10},  we  apply Lemma \ref{lem3} and Example \ref{ex4-1}
to see that
\begin{equation}\label{prop1-1}
\begin{split}
\E\left(\int_0^T |DX_k(x_t)|^2 dt\right)<\infty,\
\E\left(\int_0^T |DX_0(x_t)| dt\right)<\infty.
\end{split}
\end{equation}
In particular the integrals in the above stochastic differential equation makes sense.

Set $\bar \zeta:=\zeta \wedge \tilde \zeta$.
Applying It\^o's formula to $\bar v_t$, for every $p>2$ and any stopping time $\tau<\bar \zeta$ we obtain,
\begin{equation*}
|\bar v_{t\wedge \tau}|^p =|v|^p+\sum_{k=1}^m \int_0^{t\wedge \tau} |\bar v_s|^p dM_s+ \int_0^{t\wedge \tau} |\bar v_s|^p da_s,
\end{equation*}
where the definition of the processes $M_s$, $a_s$ are the same  as that for
$M_s^\ee$, $a_s^\ee$ by (\ref{lem9-1a}), but with $\{X_k^\ee, F_t^\ee(x), V_t^\ee(x,v)\}$ replaced
by $\{X_k, x_t, \bar v_t\}$. The estimates in  (\ref{prop1-1}) ensure that $M_s$ and $a_s$ are well defined semi-martingales.
Following the argument for (\ref{lem9-2a}), we see that
\begin{equation*}
|\bar v_{t\wedge \tau}|^p= |\bar v_0|^p \exp\left(M_{t\wedge \tau}-\frac{\langle M \rangle_{t\wedge \tau}}{2}+
a_{t\wedge \tau}\right)=0.
\end{equation*}
Thus  $v_t=\tilde v_t$ $\p$-$a.s.$ for every $t<\tau$. Since $\tau$ is arbitrary, we have $\zeta=\tilde \zeta$
$\p$-$a.s.$  and $v_t=\tilde v_t$ $\p$-$a.s.$ for every $t<\zeta$. By now we have completed the proof.

\end{proof}

\begin{theorem}\label{th1}
Suppose that Assumption \ref{assumption1} holds.
There exists a unique strong solution
$(F_t(x), V_t(x,v))$ to (\ref{sde1}) with initial value $(x,v) \in \R^{2d}$, which is defined for $t \in [0,\infty)$.
Furthermore there is a  constant $\tilde T_0>0$, such that
for every compact
set $\tilde K \subseteq \R^{2d}$,
\begin{equation}\label{th1-0}
\lim_{\ee \to 0}\sup_{(x,v) \in \tilde K}\E\left(
\sup_{t \in [0,\tilde T_0]}\left(|F_t^\ee(x)-F_t(x)|+|V_t^{\ee}(x,v)-V_t(x,v)|\right)
\right)=0.
\end{equation}
\end{theorem}
\begin{proof}
Through the proof, when the initial value $(x,v) \in \R^{2d}$ is fixed,
we denote $(F_t^\ee(x), $ $V_t^\ee(x,v))$ and $(F_t(x),V_t(x,v))$ by $(F_t^\ee, V_t^\ee)$ and $(F_t,V_t)$ respectively
for simplicity.

Since pathwise uniqueness for (\ref{sde1}) is proved in Proposition \ref{prop1}, we only need to verify  that, with  (\ref{approximated-sde}) as the approximating equations for  (\ref{sde1}),  the conditions (1)-(3) in Lemma \ref{lem5} hold.
According to Lemma \ref{lem5}, this will lead to the conclusion of the existence of a complete strong solution to (\ref{sde1})
   and the convergence in (\ref{th1-0}).

By Lemma \ref{lem10}, there exists a $\beta_1>0$, such that
for every $T>0$, compact set $K \subseteq \R^d$, $1\le k \le m$,
\begin{equation}\label{th1-1}
\begin{split}
& \sup_{\ee,\tilde \ee \in (0,\ee_0)}\sup_{x \in K}\E\left(\int_0^T \Big
|DX_k^\ee\Big(F_t^{\tilde \ee}\Big)\Big|^{2+\beta_1}dt\right)<\infty.
\end{split}
\end{equation}
For a  $\gamma_1 \in (0, \beta_1)$, let $\alpha={2+\beta_1\over 2+\gamma_1}>1$
and let $ \alpha'={2+\beta_1 \over \beta_1-\gamma_1}$ be conjugate to $\alpha$ .  By Lemma \ref{lem9}, there is a
constant $T_1(\gamma_1, \beta_1)>0$ such that for every compact set $\tilde K \subseteq \R^{2d}$,
\begin{equation}\label{th1-2}
\sup_{\ee \in (0,\ee_0)}\sup_{(x,v) \in \tilde K}\sup_{t \in [0,T_1]}
\E\left(\Big|V_t^\ee \Big|^{(2+\gamma_1)\alpha'}\right)<\infty.
\end{equation}
By H\"older inequality,
\begin{equation}\label{th1-2a}
\begin{split}
&\sup_{\ee,\tilde \ee \in (0,\ee_0)}\sup_{(x,v) \in \tilde K}
\E\left(\int_0^{T_1} \big|DX_k^\ee(F_s^{\tilde \ee})\Big( V_s^{\tilde \ee}\Big)\big|^{2+\gamma_1}ds\right)\\
&\le  \sup_{\ee,\tilde \ee \in (0,\ee_0)}\sup_{(x,v) \in \tilde K}\left\{\left(
\E\left(\int_0^{T_1} \Big|DX_k^\ee\Big(F_s^{\tilde \ee}\Big)\Big|^{2+\beta_1}ds\right)\right)^{\frac{1}{\alpha}}\right.\\
&  \qquad \qquad \qquad\left.\cdot\left(\E\left(\int_0^{T_1} \Big|V_s^{\tilde \ee}\Big|^{(2+\gamma_1)\alpha'}ds
\right)\right)^{\frac{1}{\alpha'}}\right\}<\infty.
\end{split}
\end{equation}
As the same way, there exist constants $\gamma_2>0$ and $T_2(\gamma_2,\beta_2)>0$ , such that
for every $p>0$,
\begin{equation*}
\begin{split}
&\sup_{\ee,\tilde \ee \in (0,\ee_0)}\sup_{(x,v) \in \tilde K}
\E\left(\int_0^{T_2} \big|DX_0^\ee(F_s^{\tilde \ee})\Big( V_s^{\tilde \ee}\Big)\big|^{1+\gamma_2}ds\right)<\infty,\\
&\sup_{\ee,\tilde \ee \in (0,\ee_0)}\sup_{(x,v) \in \tilde K}
\E\left(\int_0^{T_2} \big|X_k^\ee(F_s^{\tilde \ee})\big|^{p}ds\right)<\infty,\ \forall\ 0 \le k \le m.
\end{split}
\end{equation*}
Combing this with (\ref{th1-2a}) we know the condition (\ref{lem5-1}) of Lemma \ref{lem5} holds
for equation (\ref{approximated-sde}) in time interval
$t \in [0,\tilde T_0]$ with $\tilde T_0:=\min\{T_1,T_2\}$.

As the same argument above, according to Lemma \ref{lem9}
, \ref{lem10} and by H\"older inequality, we conclude that condition (\ref{lem5-2}) of Lemma \ref{lem5}
for equation (\ref{approximated-sde}) in time interval
$t \in [0,\tilde T_0]$.

We proceed to prove the last condition, condition (\ref{lem5-3}) in Lemma \ref{lem5}.
Let $\mu^{\ee,x,v}$ be the distribution of the stochastic process $(F_t^\ee(x), V_t^\ee(x,v))$ on
$\W:=C([0,\tilde T_0];$ $ \R^{2d})$ and
let $\sigma(t)=(\sigma_1(t),\sigma_2(t))$ be the canonical path on $\W$, so
the  distribution of $\sigma(\cdot)$ under $\mu^{\ee,x,v}$ is the same as that of
$(F_{\cdot}^\ee(x), V_{\cdot}^\ee(x,v) )$ under $\p$. Suppose that $\{x_n,v_n\}_{n=1}^{\infty}\subseteq \tilde K$, $\{\ee_n\}_{n=1}^{\infty}\subseteq (0,\ee_0)$ are sequences such that
$\mu^{\ee_n,x_n,v_n}$ converges weakly to some $\mu^0$  as $n \to \infty$.
By Lemma \ref{lem3+lem7}, for every $p > d+1$ and
non-negative Borel measurable function $f: \R^d \to \R_+$,
\begin{equation*}
\begin{split}
&\sup_{n}\int_{\W}\int_0^{\tilde T_0}f(\sigma_1(t))dt\;\mu^{\ee_n, x_n,v_n}(d\sigma)
\le C(\tilde K , \tilde T_0)\|f\|_{p},
\end{split}
\end{equation*}
where $\|f\|_{p}$ denotes the $L^p$ norm with respect to the Lebesgue measure.
If $f$  is furthermore bounded and  continuous,
\begin{equation}\label{th1-4}
\begin{split}
&\int_{\W}\int_0^{\tilde T_0}f(\sigma_1(t))\;dt\,\mu^{0}(d\sigma)
=\lim_{n\to \infty}\int_0^{\tilde T_0} \int_{\W}f\big(\sigma_1(t)\big)\;\mu^{\ee_n, x_n,v_n}(d\sigma)\;dt\\
&\le \sup_n \int_0^{\tilde T_0} \int_{\W}f(\sigma_1(t))\;\mu^{\ee_n, x_n,v_n}(d\sigma)\;dt
\le C(\tilde K , \tilde T_0)\|f\|_{p}.
\end{split}
\end{equation}

Let $O\subseteq \R^d$ be a bounded open set, there exists a sequence $\{g_n\}_{n=1}^{\infty}$,
of  non-negative continuous functions with compact supports
 such that
$\sup_{x \in \R^d}$ $|g_n(x)|\le 1$ and
$\lim_{n \to \infty}g_n(x)=\1_O(x)$ point wise. Then it follows from the dominated convergence theorem
that (\ref{th1-4}) holds with $f(x)=\1_O(x)$.
For every bounded measurable set $U\subseteq \R^d$ which is with null Lebesgue measure,
 from the out regularity of the Lebesgue measure, there exists a sequence of bounded open set
$\{O_n\}_{n=1}^{\infty}$ containing $U$ such that  $\lim_{n \to \infty}Leb(O_n)=0$. Then putting such
$\1_{O_n}$
into (\ref{th1-4}), letting $n \to \infty$, by Fatou lemma we have
\begin{equation}\label{th1-4a}
\int_{\W}\int_0^{\tilde T_0}\1_{U}\big(\sigma_1(t)\big)\; dt\,\mu^{0}(d\sigma)=0.
\end{equation}
Let $f: \R^d \to \R_+$ be  a non-negative  bounded Borel measurable function with
compact support. There is a sequence,
 $\{f_n\}_{n=1}^{\infty}$, of
non-negative continuous functions with compact supports  and a bounded Lebesgue-null set $Q$ such that
$\sup_n \|f_n\|_{p}$ $ \le \|f\|_{p}$ for all $1 \le p \le \infty$, and
\begin{equation}\label{th1-5a}
\lim_{n \to \infty} f_n(x)=f(x),\ \ \ \forall\ x \notin Q.
\end{equation}
It follows that
\begin{equation*}
\begin{split}
&\lim_{n \to \infty}\int_{\W}\int_0^{\tilde T_0}\Big| f_n\big(\sigma_1(t)\big)-f\big(\sigma_1(t)\big)\Big|\,dt\,\mu^{0}(d\sigma)\\
&\le \lim_{n \to \infty}\int_{\W}\int_0^{\tilde T_0}\Big(\big|f_n(\sigma_1(t))-
f(\sigma_1(t))\big|\Big)\1_{Q^c}(\sigma_1(t))\;dt\mu^{0}(d\sigma)\\
&+2\|f\|_{L_\infty} \int_{\W}\int_0^{\tilde T_0}\1_{Q}(\sigma_1(t))dt\,\mu^{0}(d\sigma)\\
&=\lim_{n \to \infty}\int_{\W}\int_0^{\tilde T_0}\big(|f_n(\sigma_1(t))-
f(\sigma_1(t))|\big)\1_{Q^c}(\sigma_1(t))\;dt\,\mu^{0}(d\sigma)=0,
\end{split}
\end{equation*}
where in the second step above we use the property (\ref{th1-4a}) and the last step is due to
(\ref{th1-5a}) and the dominated convergence theorem.
Hence putting such $f_n$ into (\ref{th1-4}) and letting $n \to \infty$, we know
(\ref{th1-4}) holds for  every non-negative bounded Borel measurable function with compact support, and
by the monotone convergence theorem, (\ref{th1-4})  holds
for every non-negative measurable function $f: \R^d \to \R_+$.

Applying (\ref{th1-4}) and Lemma \ref{lem8}, and following the proof in Lemma \ref{lem10}, for all $1 \le k \le m$ we have
\begin{equation}\label{th1-5}
\begin{split}
&\lim_{\ee \to 0}\int_{\W}\int_0^{\tilde T_0} \big|DX^\ee_k(\sigma_1(t))-
DX_k(\sigma_1(t))\big|^{2+\beta_1} \, dt \;\mu^0(d \sigma)=0.
\end{split}
\end{equation}
By (\ref{th1-2}), as the same approximation argument for (\ref{th1-4}) we can prove that
\begin{equation}\label{th1-6}
\sup_{t \in [0,\tilde T_0]}
\int_{\W}|\sigma_2(t)|^{\frac{(2+\beta_1)(2+\gamma_1)}{\beta_1-\gamma_1}}\;
\mu^0(d \sigma)<\infty.
\end{equation}
Following the same procedure for (\ref{th1-2a}), by (\ref{th1-5}), (\ref{th1-6})  and H\"older inequality
we obtain
\begin{equation*}
\begin{split}
&\lim_{\ee \to 0}\int_{\W}\int_0^{\tilde T_0} \Big|D X^\ee_k\big(\sigma_1(t)\big)
(\sigma_2(t))-
DX_k\big(\sigma_1(t)\big)(\sigma_2(t))\Big|^{2+\gamma_1}
\;dt\,\mu^0(d \sigma)=0.\\
\end{split}
\end{equation*}
Similarly, we can prove the corresponding convergence in condition (\ref{lem5-3}) of Lemma \ref{lem5}
associated with the derivative flow equation (\ref{approximated-sde}).

By now we have verified all the conditions of Lemma \ref{lem5} hold for
(\ref{approximated-sde}), so there exists a unique complete strong solution
$(F_t,V_t)$ for (\ref{sde1}) in time interval $t \in [0,\tilde T_0]$ such that
(\ref{th1-0}) holds. Let $\Phi_t(x,v,W_{\cdot}):=(F_t(x), V_t(x,v))$.
For $\tilde T_0<t\leqslant 2\tilde T_0$, we define
\begin{equation*}
 \Phi_t(x,v,W_{\cdot}):=\Phi_{t-\tilde T_0}\big(F_{\tilde T_0}(x), V_{\tilde T_0}(x,v),\theta_{\tilde T_0}(W)_{\cdot}\big),
\end{equation*}
where $\theta_{\tilde T_0}(W): C([0,\infty);\R^m) \to C([0,\infty);\R^m)$ defined by
$\theta_{\tilde T_0}(W)_t=W_{t+\tilde T_0}-W_{\tilde T_0}$ is the time shift operator.
By the Markov property and the pathwise uniqueness one may check that this is indeed the solution to
SDE (\ref{sde1}) in $t \in [\tilde T_0, 2\tilde T_0]$.
Repeating this procedure, we  will obtain a unique
global strong solution to SDE (\ref{sde1}).
\end{proof}

 \begin{remark}
In Assumption \ref{assumption1}, we assume that the elliptic constant, $|X_k|$ and
$|DX_k|$ to grow at most polynomially
as $|x| \to \infty$. The reason is that based on (\ref{c2}),
we have to apply the function $g(x):=
\log(1+|x|^2)$ in Lemma \ref{lem2} to obtain the uniform integrable property (\ref{e2}). If we strengthen
 (\ref{c2}) slightly, see Assumption \ref{assumption2} below, we may apply the polynomial function in  Lemma \ref{lem2} (see
\cite[Corollary 6.3]{Li-flow}). Moreover, following the same argument in the proof of Theorem \ref{th1} we will obtain
Corollary \ref{cor1}.
\end{remark}

 \begin{assumption}\label{assumption2}
Suppose there is a constant $\alpha \in (0, \frac{1}{2}]$ such that the following conditions are satisfied.
\begin{enumerate}
\item [(1)] There are positive constants $C_1, C_2$ such that
\begin{equation*}
\sum_{i,j=1}^d a_{i,j}(x)\xi_i \xi_j\geqslant \frac{C_1 |\xi|^2}{1+\e^{C_2|x|^{2\alpha}}} ,
\quad  \forall  x \in \R^n,\ \xi=(\xi_1,...,\xi_d)\in \R^d.
\end{equation*}
\item[(2)] There are positive constants $C_3,C_4$ such that for
all $0\le k \le m$,
\begin{equation*}
|X_k(x)|\le C_3(1+\e^{C_4|x|^{2\alpha}}).
\end{equation*}
There is a constant $\delta \in (0,1]$, and for every $p>0$ there is a constant $C(p)>0$ such that

\begin{equation*}
\begin{split}
&\sup_{|y|\le \delta}\Big(\sum_{k=1}^m p(1+|x|^{2\alpha})|X_k(x+y)|^2+\langle x , X_0(x+y)\rangle \Big)\\
& \le C(p)(1+|x|^{2(1-\alpha)}).
\end{split}
\end{equation*}

\item[(3)] Part (3) of Assumption \ref{assumption1} holds;

\item[(4)] There exists a positive constant $R_1>0$,
such that for every $p>1$,
\begin{equation*}
\begin{split}
&K_p(x)\le  C(p)(1+|x|^{2\alpha}), \ \ \ |x|>R_1,
\end{split}
\end{equation*}
for some $C(p)>0$, where the function $K_p(x)$ is defined in part (3) of the Assumption \ref{assumption1}.
Moreover, for all $0 \le k \le m$,
\begin{equation*}
|DX_k(x)|\le C_5(1+\e^{C_6|x|^{2\alpha}}),\ \ \ \ |x|>R_1,
\end{equation*}
for some positive constants $C_5$, $C_6$.
\end{enumerate}
\end{assumption}

\begin{corollary}\label{cor1}
The conclusion of Theorem \ref{th1} holds with Assumption \ref{assumption1} replaced by Assumption \ref{assumption2}.
\end{corollary}
\section{Proof of Theorem \ref{th:regularity} }\label{sec-proof}

Let $(F_t^\ee(x), V_t^\ee(x,v))$ be the solution to (\ref{approximated-sde}) with initial point
$(x,v) \in \R^{2d}$, since $X_k^\ee \in C_b^{\infty}(\R^d;\R^d)$, $x\mapsto F_t^\epsilon(x)$ is differentiable
and $V_t^\ee(x,v)=D_x F_t^\ee(x)(v)$, $\p$-a.s..  
For any given $R>0$, $p>1$, and for all $x, y \in B_R:=\{x \in \R^d;\ |x|\le R\}$, $t>0$,
\begin{equation}\label{th2-0}
\begin{split}
&\E\left(|F_t^{\ee}(x)-F_t^{\ee}(y)|^p\right)
=\E\left( \left|\left\langle x-y, \int_0^1 D_x F_t^\ee(x+s(y-x))ds \right\rangle \right|^p\right)\\
&\le C|x-y|^p\sup_{x \in B_{2R}, |v|\le 1}
\E\left(|V_t^{\varepsilon}(x,v)|^p\right).
\end{split}
\end{equation}
See also the analysis in the proof of Theorem 4.1 in \cite{Li-flow}.

According to (\ref{e2}) and Lemma \ref{lem7},  for every $p>1$, $T>0$,  and $K$ compact,
 $$\sup_{x\in K}\sup_{\ee \in (0,\ee_0)}\E\big(|F_t^{\ee}(x)|^{p+1}\big)<\infty,$$
which implies that $\{|F_t^{\ee}(x)|^{p}\}_{\ee \in (0,\ee_0), x \in K}$ is uniformly integrable.
So by Theorem \ref{th1} we derive for every $t \in [0,\tilde T_0]$,
\begin{equation*}
\lim_{\ee \to 0}\E
\left(|F_t^\ee(x)-F_t(x)|^{p}\right)=0,\ \forall x \in \R^d,
\end{equation*}
where $\tilde T_0$ is the constant in Theorem \ref{th1}.

Let $\hat T_0(p):=\min\{\tilde T_0, T_0(p)\}$, where $T_0(p)$ is the 
constant in Lemma \ref{lem9}. Therefore according to Lemma \ref{lem9}, we take the limit $\ee \to 0$
in (\ref{th2-0}) to obtain for every $t \in [0,\hat T_0]$, $x, y \in B_R$,
\begin{align*}
&\E\left(|F_t(x)-F_t(y)|^{p}\right)\\
&\le C|x-y|^p\sup_{\ee \in (0,\ee_0)}\sup_{t\in [0,\hat T_0]}\sup_{z \in B_{2R}, |v|\le 1}
\E\left(|V_t^{\varepsilon}(z,v)|^p\right)\\
&\le C(\hat T_0,R)|x-y|^p,
\end{align*}
Since $X_k$ are polynomial growth, it is easy to show for every
$0 \le s \le t \le \hat T_0(p)$, $x,y \in B_R$,
\begin{equation*}
\E\big(|F_t(x)-F_s(y)|^p\big)
\leqslant C(R,\hat T_0) \big(|x-y|^p+|t-s|^{\frac{p}{2}}\big).
\end{equation*}
In the above estimate, noting that $R$ is arbitrary large,  and we may take $p>2(d+1)$ and apply Kolmogorov's continuity criterion
 to conclude that  there
is a version of the solution flow $F_t(x,\omega)$ for SDE (\ref{sde}), such that
$F_{\cdot}(\cdot,\omega)$ is continuous in $[0,\hat T_0]\times \R^d$.

As for $t>\hat T_0$, let
$\Psi_t(x,W_{\cdot}):=F_t(x,\omega)$. By the Markov property and the uniqueness of the strong solution to SDE (\ref{sde}), it is
satisfied that
\begin{equation*}
F_t(x,\omega)=\Psi_t(x,W_{\cdot})=\Psi_{t-\hat T_0}(F_{\hat T_0}(x,\omega),\theta_{\hat T_0}(W)_{\cdot}), \ \p- a.s.
\end{equation*}
where $\theta_{\hat T_0}(W)_t=W_{t+\hat T_0}-W_{\hat T_0}$ is the time shift operator.
Hence  the solution flow $F_{\cdot}(\cdot,\omega)$ is continuous in $[0,2\hat T_0]\times \R^d$, and
  in $[0,\infty)\times \R^d$ by repeating the procedure.

Let $\{e_i\}_{i=1}^d$ be an orthonormal basis of $\R^d$ and
$(F_t(x), V_t(x,v))$ be the strong solution to (\ref{sde1})
with initial point $(x,v)\in \R^{2d}$.
By Theorem \ref{th1} and the diagonal principle there exist a subsequence
$\{\varepsilon_n\}_{n=1}^{\infty}$
with $\lim_{n\rightarrow \infty} \varepsilon_n=0$ and a set
$\tilde{\Lambda}_0 \subseteq \Omega$ with $\p(\tilde{\Lambda}_0)=0$, such that if $\omega \in \tilde{\Lambda}_0^c $,
for every $R>0$, $1\le i \le d$,
\begin{equation}\label{th2-1}
\lim_{n\rightarrow \infty}\int_{\{|x|\leqslant R\}}\sup_{ t \in [0,\tilde T_0]}
|V_t^{\ee_n}(x,e_i,\omega)-V_t(x,e_i,\omega)| dx=0,
\end{equation}
\begin{equation}\label{th2-2}
\lim_{n\rightarrow \infty}\int_{\{|x|\leqslant R\}}\sup_{ t \in [0,\tilde T_0]}
|F_t^{\varepsilon_n}(x,\omega)-F_t(x,\omega)| dx=0.
\end{equation}

For simplicity we write $(F_t^n(x),V_t^{n}(x,e_i))$ for
$(F_t^{\varepsilon_n}(x), V_t^{\varepsilon_n}(x,e_i))$. As referred above,
$D_{x}F_t^n(x)(v)=V_t^n(x,v)$ a.s.,
therefore there  exists  a $\p$-null set $\Lambda_n$,  such that
for every $\omega \in \Lambda_n^c$, the following integration by parts formula
holds for every $1\le i \le d$, $t \in [0,\tilde T_0]$ and $\varphi \in C_0^{\infty}(\R^d)$,
\begin{equation}\label{th2-3}
\int_{\R^d}{\partial \varphi \over \partial x_i}(x)  F_t^n(x,\omega)dx
=-\int_{\R^d}\varphi(x) V_t^{n}(x,e_i,\omega)dx.
\end{equation}
Let $\tilde{\Lambda}:=(\bigcup_{n=1}^{\infty}\Lambda_n)\cup\tilde{\Lambda}_0$, then $\tilde \Lambda$ is a $\p$-null set. Taking  $n$ to infinity  in (\ref{th2-3}) and using (\ref{th2-1}), (\ref{th2-2}) we see
for every $1\le i \le d$, $\omega \in \tilde{\Lambda}^c$, $t \in [0, \tilde T_0]$,
\begin{equation*}
\int_{\R^d}  {\partial \varphi \over \partial x_i}(x)F_t(x,\omega)dx
=-\int_{\R^d}\varphi(x)V_t(x,e_i,\omega)dx
\end{equation*}
which means that $F_t(\cdot,\omega)$  is weakly differentiable in
the distribution sense for  almost surely all $\omega$ and
$D_{x} F_t(x,\omega)(e_i)=V_t(x,e_i,\omega)$. Next we prove that given a $p>1$, there exist a 
$T_1>0$, such that for every $t \in [0,T_1]$,
$F_t(\cdot,\omega) \in W^{1,p}_{\loc}(\R^d;\R^d) $, a.s..

By Lemma \ref{lem9}, Theorem \ref{th1} and Fatou Lemma, given a $p>1$, there is a constant $0<T_1 \le \tilde T_0$, such that
for every $R>0$, $t \in [0, T_1]$,
\begin{equation*}
\E\left(\int_{B_R}|V_t(x,e_i)|^pdx\right)
=\int_{B_R}\E\left(|V_t(x,e_i)|^p\right)dx
\leqslant C(R,T_1).
\end{equation*}
Hence for every fixed $t \in [0,T_1]$, we can find a $\p$-null set $\Lambda_0$ (that may depend on $t$), such that
$\int_{B_R}|V_t(x,e_i,\omega)|^p dx <\infty$ for every
$R>0$, $1\le i \le d$ when $\omega \in \Lambda_0^c$.
As the same way, we can prove the similar integrable property for $F_t(x,\omega)$.
Therefore $F_t(x,\omega), V_t(x, e_i,\omega) \in L^p_{\loc}(\R^n)$ for $\omega \in
\big(\Lambda_0 \cup \tilde{\Lambda}\big)^c$. In particular,
$\Lambda:=\Lambda_0 \cup \tilde{\Lambda}$ is a
$\p$-null set. We proved that for every $t \in [0,T_1]$,
$F_t(\cdot, \omega) \in W^{1,p}_{\loc}(\R^d;\R^d)$, $\p$-a.s..

\section{The differentiation formula }
\label{bel-formula}
Suppose that Assumption \ref{assumption1} holds, let
$(F_t(x),V_t(x,v))$ be the unique strong solution of (\ref{sde1})
with initial point $(x,v) \in \R^{2d}$.
For $f \in C_b(\R^d)$ we define $P_tf(x):=\E \left(f(F_t(x))\right)$ and
let $Y:\R^d\rightarrow L(\R^d,\R^m)$  be the right inverse of map $X:\R^d\rightarrow L(\R^m,\R^d)$, where
\begin{equation}\label{J3}
X(x)(\xi):=\sum_{k=1}^ m \xi_k X_k(x)  \qquad \text{for }\ \xi=(\xi_1,\xi_2, \dots, \xi_m) \in \R^m.
\end{equation}

\begin{theorem}\label{th3}
Suppose that Assumption \ref{assumption1} holds.
There is a positive constant
$T_2$, such that for every $v \in \R^d$, $f \in C_b(\R^d)$, $t \in (0,T_2]$,
\begin{equation}\label{th3-1}
D_x (P_t f) (v)
=\frac{1}{t}\E \left(f\big(F_t(x)\big)\int_{0}^{t}
\Big\langle Y(F_s(x))(V_s(x,v)), dW_s\Big\rangle_{\R^m} \right).
\end{equation}
\end{theorem}

\begin{proof}
We first assume that $f\in C_b^1(\R^d)$. Since the coefficients of SDE (\ref{approximated-sde})
are smooth , uniformly elliptic, and with
bounded derivatives, by the classical differential
formula in \cite{Li-thesis} and \cite{Elworthy-Li-93},  we have for every $t>0$,
\begin{equation}\label{th3-2}
D_x \E \left(f(F_t^{\varepsilon}(x))\right)(v)
=\frac{1}{t}\E \left(f(F_t^{\varepsilon}(x))\int_{0}^{t}\langle
 Y^{\varepsilon}(F_s^{\varepsilon}(x))
(V_s^{\varepsilon}(x,v)), dW_s\rangle_{\R^m} \right),
\end{equation}
where $(F_t^\ee(x), V_t^\ee(x,v))$ is the strong solution to
(\ref{approximated-sde}) with initial point
$(x,v)$ $\in \R^{2d}$,
$Y^{\varepsilon}:\R^d\rightarrow L(\R^d,\R^m)$ is the right inverse of map $X^{\varepsilon}:\R^d\rightarrow L(\R^m,\R^d)$.

Since $f\in C^1_b(\R^d)$, by Theorem \ref{th1}, Lemma \ref{lem9} and H\"older inequality, there is a constant
$T_2>0$, such that for any bounded set $K$ in $\R^d$,
\begin{equation}\label{th3-3}
\lim_{\varepsilon\rightarrow 0}\sup_{x \in K}\sup_{t \in [0,T_2]}
\E\left(|f(F_t^{\varepsilon}(x))-f(F_t(x))|^8\right)=0
\end{equation}
\begin{equation}\label{th3-4}
\lim_{\varepsilon\rightarrow 0}\sup_{x \in K}\sup_{t \in [0,T_2]}
\E\left(|V_s^{\varepsilon}(x,v)-V_s(x,v)|^8\right)=0
\end{equation}

Let $A^\ee:=(X^\ee)^{\ast}X^\ee$, where $\ast$ denotes taking the transpose. Then we have
\begin{equation*}
Y^\varepsilon=(X^\varepsilon)^{\ast}(A^\ee)^{-1}.
\end{equation*}
In particular, if we write $X_k^\ee=(X^\ee_{k1},\dots,X^\ee_{kd})$,
$A^\ee=(a_{i,j}^\ee)_{i,j=1}^n$, then
$a_{i,j}^\ee=\sum_{k=1}^m X^\ee_{ki}X^\ee_{kj}$,
and for every
$\xi=(\xi_1,\dots,\xi_d) \in \R^d$,
$Y^{\varepsilon}(x)(\xi)=
(\zeta_1^{\varepsilon}(x),\zeta_2^{\varepsilon}(x)$ $...,\zeta_m^{\varepsilon}(x))$,
where $\zeta_k^{\varepsilon}(x)=\sum_{i,j=1}^d X_{ki}^{\varepsilon}(x)
b_{i,j}^{\varepsilon}(x)\xi_j$, and $(b^\ee_{i,j})=(A^\ee)^{-1}$.

By Lemma \ref{lem7},
\begin{equation*}
\sup_{\ee \in (0,\ee_0)}\big|\big(A^\ee(x)\big)^{-1}\big|\le C(1+|x|^{q})
\end{equation*}
for some $q>0$.
Combining this with (\ref{lem10-1}) and Theorem \ref{th1}, it is easy to show for every compact
set $K \subseteq \R^d$,
\begin{equation*}
\lim_{\varepsilon\rightarrow 0}\sup_{x \in K}\E\left(\int_0^{T_2}|b_{i,j}^{\varepsilon}
(F_t^{\varepsilon}(x))-b_{i,j}(F_t(x))|^8 dt\right)=0.
\end{equation*}
This together with the convergence (\ref{th3-4})
leads to
\begin{equation*}
\lim_{\varepsilon\rightarrow 0}\sup_{x \in K}\E\left(\int_0^{T_2}
|Y^{\varepsilon}(F_t^{\varepsilon}(x))(V_t^{\varepsilon}(x,v))
-Y(F_t(x))(V_t(x,v))|^4 dt\right)=0.
\end{equation*}
Then by (\ref{th3-3}) and BDG inequality, we see that for every
$t \in [0,T_2]$ and compact set $K \subseteq \R^d$,
\begin{equation*}
\begin{split}
&\lim_{\varepsilon\rightarrow 0}\sup_{x \in K}\Big|\E \left(f(F_t^{\varepsilon}(x))\int_{0}^{t}
\langle Y^{\varepsilon}(F_s^{\varepsilon}(x))(V_s^{\varepsilon}(x,v)), dW_s\rangle_{\R^m} \right)\\
&\qquad  \quad -\E \left(f(F_t(x))\int_{0}^{t}\langle Y(F_s(x))(V_s(x,v)), dW_s\rangle_{\R^m} \right)\Big|=0.
\end{split}
\end{equation*}
which implies the differentiation formula (\ref{th3-1}) holds for each
$f \in C_b^1(\R^d)$.

For $f\in C_b(\R^d)$, there is a sequence of functions $\{f_n\}_{n=1}^{\infty}\subseteq
C_b^1(\R^d)$, such that $\sup_n ||f_n||_{\infty}\le ||f||_{\infty}$,
and for every $R>0$,
\begin{equation*}
\lim_{n \to \infty}\sup_{\{|x|\le R\}}|f_n(x)-f(x)|=0.
\end{equation*}
Therefore for every $R>0$, $t \in (0,T_2]$,
\begin{equation*}
\begin{split}
&\E\left(|f_n(F_t(x))-f(F_t(x))|^2\right)\\
&\le
\sup_{\{|x|\le R\}}|f_n(x)-f(x)|^2+C||f||_{\infty}^2\p(|F_t(x)|>R)\\
&\le \sup_{\{|x|\le R\}}|f_n(x)-f(x)|^2+\frac{C||f||_{\infty}^2\E(|F_t(x)|)}{R},
\end{split}
\end{equation*}
and  by (\ref{e2}), first let $n \to 0$ and then $R \to \infty$, we obtain that for every
compact set $K \subseteq \R^d$,
\begin{equation*}
\lim_{n \to \infty}\sup_{x \in K}\E\left(|f_n(F_t(x))-f(F_t(x))|^2\right)=0,
\end{equation*}
which proves that (\ref{th3-1}) holds by standard approximation argument.
\end{proof}

\bigskip


 \end{document}